\documentclass[10pt]{amsart}
\usepackage{amssymb}
\usepackage{tikz-cd}

\usetikzlibrary{decorations.pathmorphing}

\title{Finite conjugacy classes and split exact cochain complexes}

\author {Christian Rosendal}
\address{Department of Mathematics\\University of Maryland\\4176 Campus Drive - William E. Kirwan Hall\\College Park, MD 20742-4015\\USA}
\email{rosendal@umd.edu}
\urladdr{https://sites.google.com/view/christian-rosendal/}

\newcommand{\norm}[1]{\lVert#1\rVert}
\newcommand{\Norm}[1]{\big\lVert#1\big\rVert}
\newcommand{\NORM}[1]{\Big\lVert#1\Big\rVert}
\newcommand{\triple}[1]{|\!|\!|#1|\!|\!|}

\newcommand{\forkindep}[1][]{\mathop{\mathop{\vcenter{\hbox{\oalign{\noalign{\kern-.3ex}
\hfil$\vert$\hfil\cr\noalign{\kern-.7ex}$\smile$\cr\noalign{\kern-.3ex}}}}}\displaylimits_{#1}}}

\newcommand {\Z}{\mathbb Z}

\newcommand {\R}{\mathbb R}

\newcommand {\A}{\mathbb A}
\newcommand {\F}{\mathbb F}
\newcommand {\M}{\mathbb M}

\newcommand{\om}{\omega}
\newcommand{\eps}{\epsilon}

\newcommand{\iso}{\cong}

\newcommand{\tom} {\emptyset}

\newcommand{\saa}{\Rightarrow}

\newcommand{\Lim}[1]{\mathop{\longrightarrow}\limits_{#1}}

\newcommand {\Del}{ \; \Big| \;}
\newcommand {\del}{ \; \big| \;}

\newcommand {\go} {\mathfrak}
\newcommand {\ku} {\mathcal}

\newcommand{\ov}{\overline}
\newcommand{\inv}{^{-1}}

\renewcommand {\a} {\forall}

\newcommand{\maths}[1]{\[\begin{split}{#1}\end{split}\]}

\theoremstyle{plain}
\newtheorem{thm}{Theorem}[section]
\newtheorem{cor}[thm]{Corollary}
\newtheorem{lemme}[thm]{Lemma}
\newtheorem{prop} [thm] {Proposition}
\newtheorem{defi} [thm] {Definition}

\theoremstyle{definition}

\newtheorem{rem}[thm]{Remark}
\newtheorem{exa}[thm]{Example}

\definecolor{groen}{rgb}{0,0.5,.7}
\definecolor{gul}{rgb}{0.94,0.8,0}
\definecolor{blaa}{rgb}{0.16,0,0.6}
\definecolor{roed}{rgb}{1,0,0}

\makeindex

\begin{document}

\thanks{The author was partially supported by the NSF through the award DMS 2204849 and is grateful for a number of helpful comments by N. Monod and R. Sauer.}

\maketitle

\begin{abstract}
We study the cohomology of isometric group actions on (super) reflexive Banach spaces with a focus on the relation between finite conjugacy classes and split exactness of cochain complexes. In particular, we show that, if a uniformly convex Banach module has no almost invariant vectors under the FC-centre of the acting group, then the associated cochain complex is split exact. Other similar rigidity results are established that are related to prior work of Bader, Furman, Gelander, Monod \cite{BFGM}, Bader, Rosendal, Sauer \cite{BRS} and Nowak \cite{nowak}.
\end{abstract}

\tableofcontents


\section{Introduction}

The aim of the present paper is to present a primarily algebraic approach to the study of affine isometric group actions on Banach spaces with various convexity and reflexivity properties. Whereas the prevailing approach in the literature is geometric or analytic, the algebraic viewpoint presented here affords significant simplifications to proofs and also permits one to prove results about the cohomology associated with continuous Banach group modules in all degrees. 

Let us begin by recalling that, for a group $G$, a {\em Banach $G$-module} is a pair $(X,\pi)$, where $X$ is a Banach space and $G\overset\pi\curvearrowright X$  is an action by bounded linear automorphisms. The module is furthermore said to be {\em isometric} if every operator $\pi(g)$ is an isometry of $X$. Denote by
$$
X^G=\{x\in X\del \pi(g)x=x, \a g\in G\}
$$
the closed subspace of $G$-invariant vectors. Associated with a Banach $G$-module $(X,\pi)$, one has the cochain complex
$$
C^0(G,X)\overset{\partial^1}\longrightarrow  C^1(G,X)\overset{\partial^2} \longrightarrow C^2(G,X)\overset{\partial^3} \longrightarrow C^3(G,X)\overset{\partial^4} \longrightarrow \cdots,
$$
where $C^n(G,X)$ is the vector space of maps $\prod_{i=1}^nG\overset \phi\longrightarrow X$ and the coboundary operators $C^n(G,X)\overset{\partial^{n+1}}\longrightarrow  C^{n+1}(G,X)$ are defined by the formula
\maths{
(\partial^{n+1}\phi)(g_1,\ldots, g_{n+1})=&-\pi(g_1)\phi(g_2,\ldots,g_{n+1})+(-1)^{n}\phi(g_1,\ldots, g_n)\\
&-\sum_{i=1}^n(-1)^i\phi(g_1,\ldots, g_{i-1},g_ig_{i+1}, g_{i+2},\ldots, g_{n+1}).
}
Furthermore, letting $Z^n(G,X)={\sf ker} \,\partial^{n+1}$ be the space of {\em cocycles} and $B^n(G,X)={\sf rg}\, \partial^{n}$ the space of {\em coboundaries},  the quotient 
$$
H^n(G,X)=Z^n(G,X)/B^n(G,X)
$$ 
is defined to be the {\em cohomology of degree $n$} of the cochain complex. We shall return to the interpretation of $1$-cohomology in Section \ref{affine actions}, but for now it suffices to note that $H^1(G,X)=\{0\}$ if and only if every  affine isometric action $G\overset\alpha\curvearrowright X$ with linear part $\pi$ fixes a point in $X$.

We also equip $C^n(G,X)$ with the topology of pointwise convergence on the discrete group $\prod_{i=1}^nG$ and observe that $C^n(G,X)$ is a locally convex topological vector space. In fact, when $G$ is countable,  $C^n(G,X)$ is a {\em Fr\'echet space}, that is, a locally convex, completely metrisable topological vector space.

If $F$ is any subgroup of $G$, we also let 
$$
{\sf FC}_G(F)=\{g\in G\del \text{the $F$-conjugacy class $g^F$ is finite}\}
$$
and observe that ${\sf FC}_G(F)$ is a subgroup of $G$ normalised by $F$. Of course, ${\sf FC}_G(F)$ contains the centraliser of $F$ in $G$, but will in general be larger. Let also $\Phi={\sf FC}_G(G)$ denote the {\em FC-centre} of $G$.

After presenting various background material on Banach modules, group algebra and cohomology in Sections \ref{sec:banach}, \ref{sec:algebra}, \ref{sec:fc} and \ref{sec:inv}, we begin our paper in earnest in Section \ref{sec:cohomology} by providing an explicit and simple proof of P. Nowak's result \cite{nowak} that when $(X,\pi)$ is a uniformly convex isometric Banach $G$-module without almost invariant unit vectors, then $B^1(G,X)$ is a closed complemented subspace of $C^1(G,X)$ (see Example \ref{exa: compl}). The explicit formula for the associated projection onto $B^1(G,X)$ in turn allows us to give an algebraic proof of the following result, which strengthens a theorem due to U. Bader, A. Furman, T. Gelander and N. Monod \cite{BFGM} that was proved by entirely different means. Observe that the assumption that $G=F\cdot {\sf FC}_G(F)$ generalises the case when $G$ can be written as a direct product $G=F\times M$ of subgroups $F,M\leqslant G$.

\begin{thm}
Suppose $G$ is a group and $F$ is a  subgroup so that $G=F\cdot {\sf FC}_G(F)$. Assume also that $(X,\pi)$ is a uniformly convex isometric Banach $G$-module without almost invariant unit vectors and so that $X^{{\sf FC}_G(F)}=\{0\}$. Then the restriction map 
$$
H^1(G,X)\rightarrow H^1(F,X)
$$
is zero. In other words, if $G\overset\alpha\curvearrowright X$ is an affine isometric action with linear part $\pi$, then $F$ fixes a point on $X$.
\end{thm}

Whereas Nowak's complementation result only requires the uniformly convex $G$-module to not have almost invariant unit vectors, it also solely applies to degree $1$. Nevertheless, with additional hypotheses, we can extend the complementation to all degrees. In the result below, $G'$ denotes the commutant of $G$ in the real group algebra $\R G$, whereas $\Delta G$ is the convex hull of $G$ in $\R G$. The action $G\overset\pi\curvearrowright X$ is also extended linearly to an algebra representation $\R G\overset\pi\longrightarrow \ku L(X)$, where $\ku L(X)$ denotes the Banach algebra of bounded operators on $X$.

\begin{thm}\label{intro:thm-exact}
Suppose $G$ is a group  and $(X,\pi)$ a Banach $G$-module. Assume that $\xi\in G'\cap \Delta G$ is chosen so that the operator $I-\pi(\xi)$ is invertible in $\ku L(X)$.  Then the cochain complex
$$
\{0\}\longrightarrow  C^0(G,X)\overset{\partial^1}\longrightarrow  C^1(G,X)\overset{\partial^2} \longrightarrow C^2(G,X)\overset{\partial^3} \longrightarrow C^3(G,X)\overset{\partial^4} \longrightarrow \cdots
$$
is split exact. In other words, $B^n(G,X)=Z^n(G,X)$ and $B^n(G,X)$ is complemented in $C^n(G,X)$ for all $n\geqslant 0$.
\end{thm}
Of course the major import of split exactness, as opposed to simply vanishing of the cohomology, is the fact that 
there is are continuous inverses to the coboundary operators. Thus, for example, in the setting of Theorem \ref{intro:thm-exact} we may construct a continuous linear operator $Z^2(G,X)\overset{R}\longrightarrow  C^1(G,X)$ that to each cocycle $\phi\in Z^2(G,X)$ associates a cochain $\psi=R\phi$ so that $\phi=\partial^2\psi$.

As an immediate corollary of Theorem \ref{intro:thm-exact}, we have the following result.
\begin{thm}
Suppose $G$ is a group and $(X,\pi)$ is a uniformly convex isometric Banach $G$-module. Let $\Phi$ denote the FC-centre of $G$ and assume that $X$ has no almost invariant unit vectors as a $\Phi$-module. Then the cochain complex
$$
\{0\}\longrightarrow  C^0(G,X)\overset{\partial^1}\longrightarrow  C^1(G,X)\overset{\partial^2} \longrightarrow C^2(G,X)\overset{\partial^3} \longrightarrow C^3(G,X)\overset{\partial^4} \longrightarrow \cdots
$$
is split exact.
\end{thm}

For example, this applies to all abelian or FC  groups $G$, in which case one simply has $\Phi=G$.

Similar techniques also suffice to prove vanishing of cohomology under weaker assumptions.

\begin{thm}
Suppose $G$ is a group and  $(X,\pi)$ is a uniformly convex isometric Banach $G$-module. Assume also that $F$ is a subgroup of $G$ so that $X$ has no almost invariant unit vectors as a ${\sf FC}_G(F)$-module. Then the restriction map
$$
H^n(G,X)\to H^n(F,X)
$$
is zero for all $n\geqslant 0$.
\end{thm}

In contradistinction to $Z^n(G,X)$, we observe that $B^n(G,X)$ need not, in general, be closed in $C^n(G,X)$, whereby $H^n(G,X)$ may not be Hausdorff. To counteract this, one often considers the {\em reduced cohomology} 
$$
\ov H^n(G,X)=Z^n(G,X)/\ov {B^n(G,X)},
$$
where $\ov {B^n(G,X)}$ is the closure of ${B^n(G,X)}$ in  $C^n(G,X)$. As for unreduced cohomology, $\ov H^1(G,X)$ can be interpreted in terms of affine isometric actions. Namely, $\ov H^n(G,X)=\{0\}$ if and only if every affine isometric action $G\overset\alpha\curvearrowright X$ with linear part $\pi$ has almost fixed  points in $X$ (see Section \ref{affine actions} for details). The following theorem strengthens the main result of \cite{BRS}.

\begin{thm}
Suppose that $G$ is a group and $(X,\pi)$  a separable reflexive isometric Banach $G$-module. Assume also that $F\leqslant G$ is a subgroup with $X^{{\sf FC}_G(F)}=\{0\}$. Then the restriction map
$$
\ov H^n(G,X)\to \ov H^n(F,X)
$$
is zero for all $n\geqslant 0$. In particular, if $G\overset\alpha\curvearrowright X$ is an affine isometric action with linear part $\pi$, then $F$ has almost fixed points on $X$.
\end{thm}


\section{Convexity, reflexivity and Banach modules}\label{sec:banach}
In this paper, all Banach spaces will be assumed to be real, that is, are defined over the scalar field $\R$. Nevertheless, most results hold for complex Banach spaces as well with minimal changes to the proofs.

Recall that a Banach space $(X,\norm\cdot)$ is said to be {\em strictly convex} provided that the unit sphere ${\sf S}_X=\{x\in X\del \norm x=1\}$ contains no proper line segment, or equivalently that $\norm{\frac {x+y}2}<1$ for all $x,y\in {\sf S}_X$.

Also, $(X,\norm\cdot)$ is {\em uniformly convex} if, for all $\eps>0$, there is some $\delta>0$ so that 
$$
\NORM{\frac {x+y}2}\leqslant 1-\delta
$$
whenever $x,y\in {\sf S}_X$ satisfy $\norm{x-y}\geqslant \eps$. In this case, we define the {\em modulus of uniform convexity} $\delta\colon [0,2]\to [0,1]$ by
$$
\delta(\eps)=\inf\big\{1-\NORM{\frac{x+y}2}\del x,y\in {\sf S}_X\;\&\; \norm{x-y}\geqslant \eps\big\}.
$$
Alternatively, the uniform convexity can be captured by the modulus
$$
\om(t)=\sup\{  \norm{x-y} \del \norm{x},\norm{y}\leqslant 1\;\&\; \norm{x+y}\geqslant t\}
$$
defined for $t\in [0,2]$. Indeed, $X$ is uniformly convex if and only if $\lim_{t\to 2_-}\om(t)=0$. For a uniformly convex space $X$, let us also note that, for any fixed $n$ and $\eps>0$, there is a $\delta>0$ so that
$$
{\sf diam}\big(\{x_1,\ldots,x_n\}\big)<\eps
$$
whenever $x_1,\ldots, x_n\in {\sf S}_X$ and $\Norm{\frac{x_1+\ldots+x_n}n}>1-\delta$.

In between these two concepts there is the notion of  {locally uniformly convex} spaces. Here a space $(X,\norm\cdot)$ is said to be {\em locally uniformly convex} or {\em locally uniformly rotund} if, for every $x\in {\sf S}_X$ and $\eps>0$, there is a $\delta=\delta(x,\eps)>0$ so that
$$
\NORM{\frac {x+y}2}\leqslant 1-\delta
$$
whenever $y\in {\sf S}_X$ satisfies $\norm{x-y}\geqslant \eps$. In other words, the unit sphere ${\sf S}_X$ is uniformly curved locally at every point $x\in {\sf S}_X$.

There is a very delicate and rich interplay between notions of convexity and notions of reflexivity, but we will just need to mention a couple of fundamental results. Recall first that a {\em renorming} of a Banach space $(X,\norm\cdot)$ is a norm $\triple\cdot$ on $X$ satisfying 
$$
\frac 1K\norm\cdot\leqslant \triple\cdot\leqslant K\norm\cdot
$$
for some constant $K$. In particular, this implies that the norm $\triple\cdot$ is complete on $X$ and that the formal identity $(X,\norm\cdot)\overset{\sf id}\longrightarrow (X,\triple\cdot)$ is an isomorphism of Banach spaces. Note that, since reflexivity is invariant under isomorphism, the reflexivity of $(X,\norm\cdot)$ does not change under renormings.

Secondly, a Banach space $(X,\norm\cdot)$ is said to be {\em super-reflexive} if every ultrapower $X^{\ku U}$ of $X$ is reflexive. By work of P. Enflo \cite{enflo} (see also \cite{fabian}), $(X,\norm\cdot)$ being super-reflexive is  equivalent to $(X,\norm\cdot)$ admitting a uniformly convex renorming and thus super-reflexivity is also independent of the specific choice of norm on $X$.

Suppose $G$ is a group. Then a {\em Banach $G$-module} is a pair $(X,\pi)$, where $X$ is a Banach space and $G\overset\pi\curvearrowright X$  is an action of $G$ by continuous linear automorphisms on $X$. When there is only one action $\pi$  in sight, we shall often suppress it from the discussion and simply say that $X$ is a (Banach) $G$-module.  Furthermore, if the action $\pi$ is by linear isometries of $X$, we will say that $X$ is an {\em isometric} $G$-module.

It is a well-known fact (see, for example, Proposition 2.3 \cite{BFGM} for a proof) that, if $X$ with norm $\norm\cdot$ is a super-reflexive isometric $G$-module, then there is a uniformly convex $G$-invariant renorming $\triple\cdot$ of $X$, i.e., so that $(X,\triple\cdot)$ is a uniformly convex isometric $G$-module.

Less evident is the result of G. Lancien \cite{lancien} stating that, if  $X$ with norm $\norm\cdot$ is a separable reflexive isometric $G$-module, then there is a locally uniformly convex $G$-invariant renorming $\triple\cdot$ of $X$. 
Thus, when dealing with super-reflexive or separable reflexive isometric $G$-modules $(X,\pi)$, one may after appropriate renormings of $X$, but without altering the linear action $\pi$, suppose that these are actually  uniformly convex, respectively locally uniformly convex $G$-modules.

One advantage of working with reflexive Banach modules as opposed to general Banach modules is that a reflexive  isometric Banach module decomposes canonically. In fact, this even happens in a wider context that may be independent of the geometry of the Banach space. To explain this, let us recall that an isometric Banach $G$-module $(X,\pi)$ is {\em weakly almost periodic} if, for every $x\in {\sf S}_X$, the orbit $\pi(G)x\subseteq {\sf S}_X$ is relatively weakly compact in $X$. Note that, since a Banach space is reflexive if and only if the closed unit ball is weakly compact, every isometric reflexive Banach module is automatically weakly almost periodic. The {\em Alaoglu–Birkhoff decomposition theorem} \cite{alaoglu} now states that, if $X$ is a weakly almost periodic isometric Banach $G$-module, then  $X$ admits a $G$-invariant decomposition
$$
X=X^G\oplus X_G
$$
into a direct sum of two closed linear subspaces, namely, the subspace of invariant vectors  
$$
X^G=\{x\in X\del \pi(g)x=x, \;\a g\in G\}
$$ 
and 
$$
X_G=\{x\in X\del 0\in \ov{\sf conv}(\pi(G)x)\}.
$$

Observe that, if $H$ is a normal subgroup of $G$ and $X=X^H\oplus X_H$ is the corresponding decomposition, then both $X^H$ and $X_H$ are $G$-invariant. That this holds for the space of invariant vectors is obvious and, on the other hand, if $x\in X_H$ and $g\in G$, then
$$
0=\pi(g)0\in \pi(g)\big[\ov{\sf conv}(\pi(H)x)\big] =\ov{\sf conv}(\pi(g)\pi(H)x)=\ov{\sf conv}(\pi(H)\pi(g)x),
$$
so also $\pi(g)x\in X_H$.


\section{The group algebra, affine space and simplex}\label{sec:algebra}

Suppose $G$ is a group. Then the {\em group algebra} of $G$ is the free $\R$-vector space $\R G$ over $G$, i.e., the vector space with basis $\{1_g\}_{g\in G}$. Thus, every element $\xi$ of $\R G$ has a unique representation as a linear combination 
$$
\xi=\sum_{g\in G}t_g1_{g},
$$
where, of course, only finitely many coefficients $t_g\in \R$ are non-zero. Alternatively, identifying $1_g$ with the Dirac function at $g$, $\xi$ can be viewed as a finitely supported function $G\overset\xi\longrightarrow \R$. Thus,  elements of $\R G$ add coordinatewise
$$
\sum_{g\in G}t_g1_g+\sum_{g\in G}s_g1_g=\sum_{g\in G}(t_g+s_g)1_g.
$$
Furthermore, elements of $\R G$ multiply as follows
$$
\Big(\sum_{g\in G}t_g1_g\Big)
\cdot
\Big(\sum_{g\in G}s_g1_g\Big)
=\sum_{g,f\in G}(t_gs_f)1_{gf}=\sum_{h\in G}\sum_{gf=h}(t_gs_f)1_{h}.
$$
In other words, viewing $\xi=\sum_{g\in G}t_g1_g$ and $\zeta=\sum_{g\in G}s_g1_g$ as finitely supported functions $G\overset{\xi,\zeta}\longrightarrow \R$, the product $\xi\cdot \zeta$ is just the convolution $\xi\ast \zeta$ defined by
$$
(\xi\ast \zeta)(h)=\sum_{g\in G}\xi(g)\zeta(g\inv h).
$$

We may define the {\em augmentation map} $\R G\overset{\go a}\longrightarrow \R$ by 
$$
{\go a}\Big(
\sum_{g\in G}t_g1_g
\Big)
=\sum_{g\in G}t_g
$$
and note that this is a homomorphism of $\R$-algebras. Therefore, the {\em augmentation ideal}
$$
\M G={\sf ker}\, {\go a}=\{\xi\in \R G\del \sum_{g\in G}\xi(g)=0\}
$$ 
is a two-sided ideal in $\R G$, that is, $\xi\cdot\zeta\in \M G$ whenever either $\xi\in \M G$ or $\zeta\in \M G$. Consider also the set
$$
\A G=\{\xi\in \R G\del {\go a}(\xi)=1\}=\M G+1_g,
$$
where $g\in G$ is any element. Then $\A G$ is neither closed under scalar multiplication nor under addition. On the other hand, if $\xi,\zeta\in \A G$, then 
$$
{\go a}(\xi\cdot\zeta)={\go a}(\xi)\cdot{\go a}(\zeta)=1\cdot 1=1
$$
and so $\xi\cdot\zeta\in \A G$. Similarly, if $\zeta=\sum_{i=1}^nt_i\xi_i$ is an affine combination of $\xi_i\in \A G$, that is, so that $\sum_{i=1}^nt_i=1$, then 
$$
{\go a}(\zeta)=\sum_{i=1}^nt_i{\go a}(\xi_i)=\sum_{i=1}^nt_i\cdot 1=1
$$
and so again $\zeta\in \A G$. This justifies denoting $\A G$  the {\em group affine space} of $G$.

What turns out to be of equal importance in our study is the {\em group simplex} $\Delta G$, which is the convex hull of the basis elements $\{1_g\}_{g\in G}$. In other words,
$$
\Delta G=\{\xi \in \A G\del \xi(g)\geqslant 0 \textrm{ for all }g\in G\}.
$$
Similarly to $\A G$, we see that $\Delta G$ is closed under multiplication and convex combinations.

The group $G$ naturally embeds into the multiplicative semigroup $\R G$ via the identification $g\mapsto 1_g$ and we shall therefore mostly ignore the difference between $G$ and its image in $\R G$. Therefore, the element $\sum_{g\in G}t_g1_g$ will generally simply be written $\sum_{g\in G}t_gg$.

If $F$ is a subgroup of $G$, it is of interest to determine the {\em commutant}\footnote{The notation $F'$ might unfortunately cause some confusion here, since in group theory, $F'$ usually denotes the {\em commutator subgroup} or {\em derived subgroup} of $F$, that is, the group generated by the family of all commutators $[f,g]=f\inv g\inv fg$ for $f,g\in F$.} of $F$ in the algebra $\R G$, that is, the subalgebra
$$
F'=\{\xi \in \R G\del f\cdot \xi=\xi\cdot f \textrm{ for all }f\in F\}.
$$
Obviously, the centraliser ${\sf C}_G(F)=\{g\in G\del gf=fg  \textrm{ for all }f\in F\}$ or rather its image $\{1_g\in \R G\del gf=fg  \textrm{ for all }f\in F\}$ is contained in the commutant and thus so is the linear span, ${\sf span}\big({\sf C}_G(F)\big)$. However, in general, the commutant of $F$ is larger than this. Indeed, suppose that $g\in G$ has a finite {\em $F$-conjugacy class}
$$
B=g^F=\{fgf\inv \del f\in F\}.
$$
Then we may define the {\em $F$-class sum} $\widehat B=\sum_{h\in B}h\in \R G$ and the {\em $F$-class average} $\ov B=\frac 1{|B|}\sum_{h\in B}h\in \Delta G$. 

\begin{lemme}\label{commutant}
Let $F$ be a subgroup of $G$. Then the commutant of $F$ in $\R G$ is
$$
F'={\sf span}\big\{ \widehat B\del B\textrm{ is a finite $F$-conjugacy class in }G\big\}.
$$
Furthermore, $F'\cap \Delta  G$ is a multiplicative subsemigroup of $\R G$  and 
$$
F'\cap \Delta G={\sf conv}\big\{ \ov B\del B\textrm{ is a finite $F$-conjugacy class in }G\big\}.
$$
\end{lemme}

\begin{proof}
Note first that, if $B$ is a finite $F$-conjugacy class, then
$$
f\cdot \widehat B\cdot f\inv =\sum_{h\in B}fhf\inv=\sum_{h\in B}h=\widehat B
$$
and so $\widehat B\in F'$. Conversely,  suppose that $\xi=\sum_{g\in G}t_gg\in F'$ and that $f\in F$. Then 
$$
\sum_{h\in G}t_hh=\xi=f\cdot \xi\cdot f\inv=\sum_{g\in G}t_gfgf\inv=\sum_{h\in G}t_{f\inv hf}h 
$$
and so $\xi(h)=t_h=t_{f\inv hf}=\xi(f\inv hf)$ for all $h\in G$. It follows that $\xi$ is constant on every $F$-conjugacy class and therefore must be a linear combination of $F$-class sums. 

Note that this also shows that 
$$
F'\cap \Delta G={\sf conv}\big\{\ov B\del B\textrm{ is a finite $F$-conjugacy class in }G\big\}.
$$
Finally, because both $F'$ and $\Delta G$ are closed under multiplication, so is their intersection. That is, $F'\cap \Delta$ is a multiplicative subsemigroup of the algebra $\R G$.
\end{proof}

One last thing to note is that, if $G$ is a group and $(X,\pi)$ a Banach $G$-module, then the module action $G\overset\pi\curvearrowright X$ extends canonically to a module action of the algebra $\R G$. More precisely, seeing $\pi$ as a map $G\overset\pi\longrightarrow \ku L(X)$ into the algebra $\ku L(X)$ of bounded linear operators on $X$, $\pi$ extends uniquely to an algebra representation
$$
\R G\overset\pi\longrightarrow \ku L(X)
$$
simply by setting $\pi\big(\sum_it_ig_i\big)=\sum_it_i\pi(g_i)$.


\section{Groups with finite conjugacy classes}\label{sec:fc}

\begin{defi}
Suppose $F$ is a subgroup of a group $G$. We then let 
$$
{\sf FC}_G(F)=\{g\in G\del \text{the $F$-conjugacy class } g^F \text{ is finite}\}.
$$
\end{defi}

The basic observation in this context is that ${\sf FC}_G(G)$ is a subgroup of $G$ that is evidently normalised by $F$, i.e., $F\leqslant {\sf N}_G\big({\sf FC}_G(F)\big)$. Indeed, if $g^F=\{h_1,\ldots, h_n\}$, then $(g\inv)^F=\{h_1\inv,\ldots, h_n\inv\}$. Also,  if $g_1,\ldots, g_n\in {\sf FC}_G(F)$, then 
$$
(g_1\cdots g_n)^F\subseteq g_1^F\cdots g_n^F
$$
and the latter set is a finite product of finite sets and therefore finite. Thus $g_1\cdots g_n\in {\sf FC}_G(F)$

Observe also that,  ${\sf FC}_G(F)$ consists exactly of the $g\in G$ so that the centraliser subgroup ${\sf C}_F(g)=\{f\in F\del fg=gf\}$ has finite index in $F$. In particular, if $H$ is a finitely generated subgroup of ${\sf FC}_G(F)$, then the centraliser 
$$
{\sf C}_F(H)=\{f\in F\del fh=hf, \a h\in H\}
$$ 
is also of finite index in $F$. If $H$ is no longer finitely generated, this may no longer be the case.

Recall that, if $G$ is a group, the {\em FC-centre} of $G$ is the normal subgroup 
$$
\Phi={\sf FC}_G(G) =\{g\in G\del g \text{ has a finite conjugacy class } \}
$$ 
and $G$ is called an {\em FC-group} in case $G=\Phi$, i.e., if every conjugacy class is finite. Note that, if $G$ is any group and $F$ is a subgroup, then $K=F\cdot {\sf FC}_G(F)$ is also a subgroup of $G$, but the FC-centre of $K$ may in general be smaller than ${\sf FC}_G(F)$, since ${\sf FC}_G(F)$ may not itself be an FC-group.

\begin{exa}
Let $X_1, X_2, X_3, \ldots$ be a sequence of disjoint finite sets and let $\F_\infty$ be the free group on the set of generators $\bigcup_nX_n$. Note that $\prod_n{\rm Sym}(X_n)$ acts by automorphisms on $\F_\infty$ by permuting the generators. Hence, if $\Gamma$ is any subgroup of $\prod_n{\rm Sym}(X_n)$, we can form the semidirect product $G=\F_\infty\rtimes \Gamma$, in which $\F_\infty$ is normal and $\F_\infty\leqslant {\sf FC}_G(\Gamma)$. In particular, $G=\Gamma\cdot {\sf FC}_G(\Gamma)$.
\end{exa}

\begin{lemme}\label{X^F'}
Suppose that $G$ is a group, that $F$ is a subgroup of $G$ and that $(X,\pi)$ is a strictly convex isometric Banach $G$-module. Then
$$
X^{{\sf FC}_G(F)}=\{x\in X\del \pi(\xi)x=x \textrm{ for all }\xi\in F'\cap \Delta G\}.
$$
\end{lemme}

\begin{proof}
Suppose first that $x\in X^{{\sf FC}_G(F)}$. To see that $\pi(\xi)x=x$ for all $\xi\in F'\cap \Delta G$,  it suffices, by Lemma \ref{commutant}, to show that $\pi\big(\ov B\big)x=x$ for all finite $F$-conjugacy classes $B$. But such $B$ are subsets of ${\sf FC}_G(F)$ and therefore
$$
\pi(\ov B)x=\frac 1{|B|}\sum_{g\in B}\pi(g)x=\frac 1{|B|}\sum_{g\in B}x=x.
$$

Conversely, suppose $x\in X$ is a unit vector so that $\pi(\xi)x=x$ for all $\xi\in F'\cap \Delta G$. Assume also that $h\in {\sf FC}_G(F)$ and let $B=h^F$ be the $F$-conjugacy class of $h$. Then, by Lemma \ref{commutant},  $\ov B\in F'\cap \Delta G$ and therefore
$$
x=\pi(\ov B)x=\frac 1{|B|}\sum_{g\in B}\pi(g)x.
$$
Because $\norm{\pi(g)x}=1$ for all $g\in G$ and $X$ is strictly convex, we find that $x=\pi(g)x$ for all $g\in B$. In particular, $\pi(h)x=x$, showing that $x\in X^{{\sf FC}_G(F)}$.  
\end{proof}


\section{Almost invariant unit vectors}\label{sec:inv}

\begin{defi}
Let $G$ be a group and  $(X,\pi)$ an isometric Banach $G$-module. We say that $(X,\pi)$ has {\em  almost invariant unit vectors} if, for all finite subsets $E\subseteq G$ and $\eps>0$, there is some $x\in {\sf S}_X$ so that
$$
\norm{x-\pi(g)x}<\eps
$$
for all $g\in E$.
\end{defi}

\begin{lemme}\label{invertibility}
Let $G$ be a group and  $(X,\pi)$ a uniformly convex isometric Banach $G$-module. Then the  following conditions are equivalent,
\begin{enumerate}
\item $(X,\pi)$  does not have almost invariant unit vectors,
\item the operator $I-\pi(\xi)$ is invertible for some $\xi \in \Delta G$,
\item $\norm{\pi(\xi)}<1$ for some $\xi\in \Delta G$,
\item there is a finite set $E\subseteq G$ so that $\norm{\pi(\xi)}<1$ for all $\xi\in \Delta G$  with $E\subseteq {\sf supp}(\xi)$. 
\end{enumerate}
\end{lemme}

\begin{proof}
Clearly, (4)$\saa$(3).  That (3)$\saa$(2) follows from considering the Carl Neumann series 
$$
\big(I-\pi(\xi)\big)\inv=\sum_{n=0}^\infty\pi(\xi)^n,
$$
which is valid whenever $\norm{\pi(\xi)}<1$. Similarly, that (2)$\saa$(1) is clear. Because, if $I-\pi(\xi)$ is invertible, then it is bounded away from $0$, that is, 
$$
\norm{x-\pi(\xi)x}=\Norm{\big(I-\pi(\xi)\big)x}\geqslant \eps \norm x
$$ 
for some $\eps>0$ and all $x\in X$. So, if $\xi=\sum_{i=1}^k\lambda_ig_i$, then no vector $x\in {\sf S}_X$ can satisfy $\norm{x-\pi(g_i)x}<\eps$ for all $i$.

Finally, to see that (1)$\saa$(4), suppose $g_1,\ldots, g_n\in G$ and $\eps>0$ are chosen so that no unit vector $x$ satisfies $\norm{x-\pi(g_i)x}<\eps$ for all $i$. This means that, if $E=\{1, g_1,\ldots, g_n\}$, then
$$
{\sf diam}\big(\{\pi(g)x\del g\in E\}\big)\geqslant \eps
$$
for all $x\in {\sf S}_X$. Assume now that $\xi \in \Delta G$ satisfies $E\subseteq {\sf supp}(\xi)$. Then, if $x\in {\sf S}_X$, 
$$
\pi(\xi)x=\sum_{g\in G}\xi(g)\pi(g)x
$$ 
is a convex combination of a subset of ${\sf S}_X$ of radius at least $\eps$, of cardinality at most $m=\big|{\sf supp}(\xi)\big|$ and with minimal positive coefficient greater than or equal to $\lambda=\min\{\xi(g)\del \xi(g)>0\}$. Because $X$ is  uniformly convex, it follows that there is some $\delta=\delta(\eps,m,\lambda)>0$ so that
$$
\norm{\pi(\xi)x}<1-\delta
$$
for all $x\in {\sf S}_X$, that is, $\norm{\pi(\xi)}\leqslant1-\delta$.
\end{proof}

\begin{lemme}\label{existence of xi zeta}
Let $G$ be a group and $F$ a subgroup so that $G=F\cdot {\sf FC}_G(F)$. Assume also that $(X,\pi)$ is a uniformly convex isometric Banach $G$-module without almost invariant unit vectors.
 Then 
$$
\norm{\pi(\xi)\pi(\zeta)}<1
$$ 
for some $\xi\in \Delta F$ and $\zeta\in F'\cap \Delta G$.
\end{lemme}

\begin{proof}
By Lemma \ref{invertibility} there is a finite set $E\subseteq G$ so that $\norm{\pi(\xi)}<1$ for all $\xi\in \Delta G$  with $E\subseteq {\sf supp}(\xi)$. Also, as $G=F\cdot {\sf FC}_G(F)$, we can find finite sets $A\subseteq F$ and $B\subseteq {\sf FC}_G(F)$ so that $E\subseteq AB$. Enlarging $B$ if necessary, we may assume that $B$ is a union of $F$-conjugacy classes. Letting $\xi=\frac 1{|A|}\sum_{f\in A}f\in \Delta F$ and $\xi=\frac 1{|B|}\sum_{g\in B}g$, we find that
$$
\xi\zeta=\frac 1{|A|\cdot |B|}\sum_{f\in A}\sum_{g\in B}fg
$$
and so $E\subseteq AB\subseteq{\sf supp}\,(\xi\zeta)$. It follows that $\norm{\pi(\xi)\pi(\zeta)}=\norm{\pi(\xi\zeta)}<1$. Observe also that, by Lemma \ref{commutant}, we have $\xi\in F'\cap \Delta G$.
\end{proof}


\section{Cohomology}\label{sec:cohomology}
For the following discussion, consider a group $G$ and a Banach $G$-module $(X,\pi)$. This gives rise to the standard {\em cochain complex}, which we define as follows. First, for every $n\geqslant 0$, let
$$
C^n(G,X)
$$
denote the vector space of {\em $n$-cochains}, that is, $C^n(G,X)$ is the collection of all maps $G^n\overset \phi\longrightarrow X$. Observe that, because $G^0=\{\tom\}$, every $\phi \in C^0(G,X)$ can be identified with its unique value $\phi(\tom)\in X$ and we can therefore identify $C^0(G,X)$ with $X$ itself. Define now a sequence of linear operators
$$
\{0\}\overset{\partial^0}\longrightarrow C^0(G,X)\overset{\partial^1}\longrightarrow C^1(G,X)\overset{\partial^2}\longrightarrow C^2(G,X)\overset{\partial^3}\longrightarrow \ldots
$$
by the formula 
\maths{
(\partial^{n+1}\phi)(g_1,\ldots, g_{n+1})=&-\pi(g_1)\phi(g_2,\ldots,g_{n+1})+(-1)^{n}\phi(g_1,\ldots, g_n)\\
&-\sum_{i=1}^n(-1)^i\phi(g_1,\ldots, g_{i-1},g_ig_{i+1}, g_{i+2},\ldots, g_{n+1}).
}
In particular, for $x\in X=C^0(G,X)$, we have 
$$
(\partial^1x)(g)=x-\pi(g)x,
$$
whereas, for $\phi\in C^1(G,X)$, we have
$$
(\partial^2\phi)(g,f)=\phi(gf)-\pi(g)\phi(f)-\phi(g).
$$
A straightforward computation shows that $\partial^{n+1}\circ\partial^n=0$. Also, when no confusion is possible, we shall often drop the index $n$ from $\partial^n$.

For all $n\geqslant 0$, let 
$$
Z^n(G,X)={\sf ker}\;\partial^{n+1}
$$ 
denote the space of {\em $n$-cocycles} and
$$
B^{n}(G,X)={\sf rg}\;\partial^{n}
$$
the space of {\em $n$-coboundaries}. From the above, note that $B^0(G,X)=\{0\}$ and $Z^0(G,X)=X^G$, whereas
$$
Z^1(G,X)=\{\phi\colon G\to X\del \phi(gf)=\pi(g)\phi(f)+\phi(g) \text{ for all }g,f\in G\}.
$$
For $n\geqslant 0$, the {\em $n$-cohomology} of the $G$-Banach module $(X,\pi)$ is defined to be the quotient vector space 
$$
H^n(G,X)=\frac{Z^n(G,X)}{B^n(G,X)}.
$$
We shall return to its interpretation in Section \ref{affine actions}.

Observe also that $C^n(G,X)$ can be identified with the product space $\prod_{G^n}X$ and therefore is equipped with the corresponding Tychonoff product topology. Furthermore, $Z^n(G,X)$ is a closed linear subspace of $C^n(G,X)$, because
$$
Z^n(G,X)=\bigcap_{(g_1,\ldots,g_{n+1})\in G^n}\Big\{\phi\in \prod_{G^n}X\Del (\partial^{n+1}\phi)(g_1,\ldots,g_{n+1})=0\Big\}.
$$
In particular, when $G$ is a countable group, both $C^n(G,X)$ and $Z^n(G,X)$ are Fr\'echet spaces.

\begin{exa}[Complementation of $1$-cohomology]\label{exa: compl}
Observe that, for any element $\xi =\sum_it_ig_i\in \Delta G$, the operator $I-\pi(\xi)$ on $X$ factors through $C^1(G,X)$ as follows
$$
\begin{tikzcd}
X\arrow{dr}{\partial^1}\arrow{rr}{I-\pi(\xi)}      & &       X      \\
 &C^1(G,X)\arrow{ur}{R_\xi} &\\
\end{tikzcd}
$$
where $R_\xi$ is the operator given by 
$$
R_\xi(\phi)=\sum_it_i\phi(g_i).
$$ 
Indeed, simply note that
$$
\big(R_\xi\circ \partial^1\big)(x)
=\sum_it_i\big(\partial^1x\big)(g_i)
=\sum_it_i\big(x-\pi(g_i)x\big)
=x-\pi(\xi)x.
$$

Suppose now that $\xi$ is chosen so that $I-\pi(\xi)$ is invertible. For example, by Lemma \ref{invertibility}, this can be achieved when $X$ is uniformly convex and has no almost invariant unit vectors. In this case,
$$
P=\partial^1\circ \big(I-\pi(\xi)\big)\inv \circ R_\xi
$$
defines a continuous idempotent operator on the topological vector space $C^1(G,X)$, because
$$
P^2=\partial^1\circ \big(I-\pi(\xi)\big)\inv \circ R_\xi\circ \partial^1\circ \big(I-\pi(\xi)\big)\inv \circ R_\xi=\partial^1\circ \big(I-\pi(\xi)\big)\inv \circ R_\xi=P.
$$
It follows that $P$ is a continuous projection onto its image $B^1(G,X)$ and also that $B^1(G,X)={\sf ker}\,(I-P)$ is a closed subspace of $C^1(G,X)$. Note also that, because $I-\pi(\xi)$ is invertible, no unit vector is invariant under $\pi(\xi)$ and so $X^G=\{0\}$ and $\partial^1$ is injective. Therefore, ${\sf ker}\,P={\sf ker}\,R_\xi$ and the space of cochains decomposes as a topological direct sum
$$
C^1(G,X)={\sf ker}\,(I-P)\oplus{\sf ker}\,P=B^1(G,X)\oplus {\sf ker}\,R_\xi.
$$
Furthermore, because $B^1(G,X)\subseteq Z^1(G,X)$, this shows that $B^1(G,X)$ is complemented in $Z^1(G,X)$ and hence also that 
$$
Z^1(G,X)\,\iso\, B^1(G,X)\oplus H^1(G,X).
$$
This is the main result of P. Nowak's paper \cite{nowak}.
\end{exa}


\section{Extension of cohomology to the group affine space}
Note that, because $G$ is a vector space basis for its group algebra $\R G$, every map $G^n\overset\phi\longrightarrow X$ has a unique extension to an $n$-multilinear map
$$
\underbrace{\R G\times \ldots\times \R G}_{n \textrm{ times}}\overset{\tilde\phi}\longrightarrow X.
$$
Thus, the vector space $C^n(G,X)$ is canonically isomorphic to the vector space $C^n(\R G,X)$ of $n$-multilinear maps $(\R G)^n\overset\psi\longrightarrow X$. However, the extension operations $\phi\mapsto\tilde \phi$ fail to commute with the coboundary homomorphisms $\partial$ (see Example \ref{failure} below) and we must therefore instead restrict the attention to a smaller space of functions on which they do commute.\footnote{It is possible to extend cohomology to all of the group algebra $\R G$. However, this extension involves the augmentation map and will not be pursued here.}  Indeed, for a map $G^n\overset\phi\longrightarrow X$, let $\hat \phi$ denote the unique extension of $\phi$ to an {\em $n$-multiaffine map}
$$
\underbrace{\A G\times \ldots\times \A G}_{n \textrm{ times}}\overset{\hat\phi}\longrightarrow X,
$$
that is, so that
$$
\hat \phi\big(\xi_1,\ldots, \xi_{i-1}, \sum_{j=1}^kt_j\zeta_j,\xi_{i+1},\ldots, \xi_n\big)= \sum_{j=1}^kt_j\cdot\hat \phi(\xi_1,\ldots, \xi_{i-1}, \zeta_j,\xi_{i+1},\ldots, \xi_n)
$$
for all $\xi_l,\zeta_j\in \A G$ and $t_l\in \R$ with $\sum_{j=1}^kt_j=1$. Let then $C^n(\A G,X)$ denote the collection of all $n$-multiaffine maps $\A G\times \ldots\times \A G\overset{\psi}\longrightarrow X$ and note that $C^n(G,X)$ and $C^n(\A G, X)$ are canonically isomorphic via the extension map $\phi\mapsto \hat\phi$.

Finally, define the {\em coboundary maps}
$$
C^n(\A G,X)\;\overset{\partial^{n+1}}\longrightarrow C^{n+1}(\A G,X)
$$
by the formula
\maths{
(\partial^{n+1}\phi)(\xi_1,\ldots, \xi_{n+1})
=&-\pi(\xi_1)\phi(\xi_2,\ldots,\xi_{n+1})
+(-1)^{n}\phi(\xi_1,\ldots, \xi_n)\\
&-\sum_{i=1}^n(-1)^i\phi(\xi_1,\ldots, \xi_{i-1},\xi_i\xi_{i+1}, \xi_{i+2},\ldots, \xi_{n+1}).
}

The following lemma now shows that, as opposed to the case of $C^n(\R G,X)$, the identification of $C^n(G,X)$ with $C^n(\A G,X)$ is compatible with the coboundary homomorphism $\partial^{n+1}$.
\begin{lemme}
For all $n\geqslant 0$ and $\phi\in C^n(G,X)$, we have $\partial(\hat\phi)=\widehat{\partial\phi}$ and so the diagram
$$
\begin{tikzcd}
C^n(G,X)\arrow{d}{\iso}\arrow{r}{\partial^{n+1}}      & C^{n+1}(G,X)\arrow{d} {\iso}  \\
C^n(\A G,X)\arrow{r}{\partial^{n+1}}  & C^{n+1}(\A G,X) \\
\end{tikzcd}
$$
commutes.
\end{lemme}

\begin{proof}
Suppose that $\xi_1,\ldots,\xi_{n+1}\in \A G$. Then we can find $g_1,\ldots, g_k\in G$ and $t_{i,j}\in \R$ with  $\sum_{j=1}^kt_{i,j}=1$, so that  $\xi_i=\sum_{j=1}^kt_{i,j}g_j$ for all $i$. We then have
\maths{
\widehat{\partial\phi}(&\xi_1,\ldots,\xi_{n+1})
=\sum_{j_1,\ldots,j_{n+1}}t_{1,j_1}\cdots t_{n+1,j_{n+1}}\partial\phi(g_{j_1},\ldots,g_{j_{n+1}})\\
=&-\sum_{j_1,\ldots,j_{n+1}} t_{1,j_1}\cdots t_{n+1,j_{n+1}}\pi(g_{j_1})\phi(g_{j_2},\ldots,g_{j_{n+1}})\\
&-\sum_{j_1,\ldots,j_{n+1}} t_{1,j_1}\cdots t_{n+1,j_{n+1}}\sum_{p=1}^n(-1)^p\phi(g_{j_1},\ldots,g_{j_{p-1}},g_{j_p}g_{j_{p+1}},g_{j_{p+2}},\ldots,g_{j_{n+1}})\\
&+(-1)^{n}\sum_{j_1,\ldots,j_{n+1}} t_{1,j_1}\cdots t_{n+1,j_{n+1}}\phi(g_{j_1},\ldots,g_{j_{n}})\\
=&-\sum_{j_1} t_{1,j_1}\pi(g_{j_1})\hat\phi(\xi_2,\ldots,\xi_{n+1})\\
&-\sum_{p=1}^n(-1)^p\sum_{j_p} t_{p,j_p}\hat\phi(\xi_1,\ldots,\xi_{p-1},g_{j_p}\xi_{p+1},\xi_{p+2},\ldots,\xi_{n+1})\\
&+(-1)^{n}\sum_{j_{n+1}}t_{n+1,j_{n+1}}\hat\phi(\xi_1,\ldots,\xi_n)\\
=&-\pi(\xi_1)\hat\phi(\xi_2,\ldots,\xi_{n+1})\\
&-\sum_{p=1}^n(-1)^p\hat\phi(\xi_1,\ldots,\xi_{p-1},\xi_p\xi_{p+1},\xi_{p+2},\ldots,\xi_{n+1})\\
&+(-1)^{n}\hat\phi(\xi_1,\ldots,\xi_n)\\
=&\partial(\hat\phi)(\xi_1,\ldots,\xi_{n+1}).
}
This proves the lemma.
\end{proof}

\begin{rem}\label{no worries}
By this lemma, we need not worry about the difference between the two spaces $C^n(G,X)$ and $C^n(\A G, X)$. For example, if we define $Z^n(\A G,X)={\sf ker}\,\partial^{n+1}$ and $B^n(\A G,X)={\sf rg}\,\partial^n$, then the extension map $\phi\mapsto \hat\phi$ defines an isomorphism between $Z^1(G,X)$ and $Z^1(\A G,X)$, respectively between $B^1(G,X)$ and $B^1(\A G,X)$. Concretely, this implies among other things that, if $\phi\in Z^1(G,X)$ and $\xi,\zeta\in \A G$, then 
$$
\hat\phi(\xi\zeta)=\pi(\xi)\hat\phi(\zeta)+\hat\phi(\xi),
$$
because $\partial^2\hat\phi=\widehat{\partial^2\phi}=0$. 

Henceforth, we shall not worry about the distinction between $\phi$ and its extension $\hat \phi$ to  an $n$-multiaffine map and simply denote both by $\phi$. 
\end{rem}

\begin{rem}\label{commutativity}
Suppose $\xi, \zeta\in \Delta G$ commute, that is $\xi\zeta=\zeta\xi$, and that $\phi\in Z^1(G,X)$.
Then
$$
\pi(\xi)\phi(\zeta)+\phi(\xi)=
\phi(\xi\zeta)=\phi(\zeta\xi)=\pi(\zeta)\phi(\xi)+\phi(\zeta)
$$
and so
$$
\big(I-\pi(\zeta)\big)\phi(\xi)=\big(I-\pi(\xi)\big)\phi(\zeta)
$$
or equivalently
$$
\partial^1\big(\phi(\xi)\big)(\zeta)=\partial^1\big(\phi(\zeta)\big)(\xi).
$$
We shall be using this simple observation repeatedly.
\end{rem}

\begin{exa}\label{failure}
To see that the standard formulas fail to  extend the cochain complex to the group algebra $\R G$ rather than just $\A G$, suppose that $C^n(\R G,X)\;\overset{\partial^{n+1}}\longrightarrow C^{n+1}(\R G,X)$ is given as before by the formula
by the formula
\maths{
(\partial^{n+1}\phi)(\xi_1,\ldots, \xi_{n+1})
=&-\pi(\xi_1)\phi(\xi_2,\ldots,\xi_{n+1})
+(-1)^{n}\phi(\xi_1,\ldots, \xi_n)\\
&-\sum_{i=1}^n(-1)^i\phi(\xi_1,\ldots, \xi_{i-1},\xi_i\xi_{i+1}, \xi_{i+2},\ldots, \xi_{n+1}).
}
Then, for $x\in X=C^0(G,X)$ and $g,f\in G$, we have
$$
(\partial \tilde x)(g+f)=-\pi(g+f)\tilde x+\tilde x=x-\pi(g)x-\pi(f)x,
$$
whereas
$$
\widetilde{(\partial x)}(g+f)=(\partial x)(g)+(\partial x)(f)=x-\pi(g)x+x-\pi(f)x.
$$
So $\widetilde{(\partial x)}\neq \partial \tilde x$ whenever $x\neq 0$.
\end{exa}

Suppose $G$ is a group, $F$ is a subgroup and $(X,\pi)$ is a Banach $G$-module. Then the map that takes a cochain $\phi\in C^n(G,X)$ to its restriction $F^n\overset{\phi|_F}\longrightarrow X$ is easily seen to commute with the coboundary map $\partial$ and therefore maps $Z^n(G,X)$ into $Z^n(F,X)$ and $B^n(G,X)$ into $B^n(F,X)$. From this it also follows that the restriction map $\phi\mapsto \phi|_F$ induces a map on cohomology $H^1(G,X)\rightarrow H^1(F,X)$.

Our first result is a refinement of Theorem C \cite{BFGM}, which was proved using very different means, namely minimal invariant closed convex sets. The proof here, based on Example \ref{exa: compl}, is purely algebraic and somewhat shorter.
\begin{thm}\label{zero restriction}
Suppose $G$ is a group and $F$ is a  subgroup so that $G=F\cdot {\sf FC}_G(F)$. Assume also that $(X,\pi)$ is a uniformly convex isometric Banach $G$-module without almost invariant unit vectors and so that $X^{{\sf FC}_G(F)}=\{0\}$. Then the restriction map 
$$
H^1(G,X)\rightarrow H^1(F,X)
$$
is zero.
\end{thm}

\begin{proof}
Because $X$ is uniformly convex and has no almost invariant unit vectors, it follows from Lemma \ref{existence of xi zeta} that there are $\xi \in \Delta F$ and $\zeta\in F'\cap \Delta G$ so that $\norm{\pi(\xi)\pi(\zeta)}<1$. It now follows from Example \ref{exa: compl} that $C^1(G,X)$  decomposes as a direct sum of closed linear subspaces 
$$
C^1(G,X)=B^1(G,X)\oplus \{\phi\in C^1(G,X)\del \phi(\xi\zeta)=0\}.
$$
To prove the theorem, it thus suffices to show that $\phi|_F=0$ for all  $\phi\in Z^1(G,X)$ satisfying $\phi(\xi\zeta)=0$. 

So let such a $\phi$ be given. Then
$$
\pi(\xi)\phi(\zeta)+\phi(\xi)=\phi(\xi\zeta)=0=\phi(\zeta \xi)=\pi(\zeta)\phi(\xi)+\phi(\zeta)
$$
and so 
$$
\phi(\zeta)=-\pi(\zeta)\phi(\xi)=-\pi(\zeta)\big[-\pi(\xi)\phi(\zeta)\big]=\pi(\xi\zeta)\phi(\zeta).
$$ 
As $\norm{\pi(\xi\zeta)}<1$, it follows that $\phi(\zeta)=0$ and similarly $\phi(\xi)=0$. 

Observe now that, for any $\sigma\in \Delta F$, we have
$$
\big(I-\pi(\zeta)\big)\phi(\sigma)=\big(I-\pi(\sigma)\big)\phi(\zeta)=0.
$$
Thus $\pi(\zeta)\phi(\sigma)=\phi(\sigma)$ and
$$
\phi(\xi\sigma)=\pi(\xi)\phi(\sigma)+\phi(\xi)=\pi(\xi)\phi(\sigma)=\pi(\xi\zeta)\phi(\sigma).
$$
Therefore, if $\eta\in \Delta F$, we have 
\maths{
\phi(\xi^n\eta)=\phi(\xi\cdot\xi^{n-1}\eta)=\pi(\xi\zeta)\phi(\xi^{n-1}\eta)=\ldots=\pi(\xi\zeta)^n\phi(\eta)
}
and so $\phi(\xi^n\eta)\Lim n 0$ because $\norm{\pi(\xi\zeta)}<1$.

Assume that $f\in F$ and $\beta\in F'\cap \Delta G$. Set $\eta=\frac{1+f}2\in \Delta F$ and note that
\maths{
\norm{\phi(\beta)}
&\geqslant \norm{\pi(\eta)\phi(\beta)}\\
&\geqslant \norm{\pi(\xi^n\eta)\phi(\beta)}\\
&=\norm{\phi(\xi^n\eta\beta)-\phi(\xi^n\eta)}\\
&=\norm{\phi(\beta\xi^n\eta)-\phi(\xi^n\eta)}\\
&=\Norm{\phi(\beta)-\big(I-\pi(\beta)\big)\phi(\xi^n\eta)}\\
&\Lim n\norm{\phi(\beta)}.
}
It thus follows that
$$
\NORM{\frac {\phi(\beta)+\pi(f)\phi(\beta)}2}=\norm{\pi(\eta)\phi(\beta)}=\norm{\phi(\beta)}
$$
and so $\pi(f)\phi(\beta)=\phi(\beta)$ because $X$ is strictly convex. Finally,
$$
\big(I-\pi(\beta)\big)\phi(f)=\big(I-\pi(f)\big)\phi(\beta)=0
$$
and so $\pi(\beta)\phi(f)=\phi(f)$.
Because $\beta\in F'\cap \Delta G$ was arbitrary,  Lemma \ref{X^F'} implies that $\phi(f)\in X^{{\sf FC}_G(F)}=\{0\}$.
In other words, $\phi|_F=0$ and so the restriction mapping $H^1(G,X)\rightarrow H^1(F,X)$ is zero.
\end{proof}

The following corollary is a reformulation of Theorem C \cite{BFGM}.
\begin{cor}
Suppose $F$ and $G$ are groups and $(X,\pi)$ is a uniformly convex isometric Banach $F\times G$-module without almost invariant unit vectors. Assume also that $X^F=X^G=\{0\}$. Then 
$$
H^1(F\times G,X)=\{0\}.
$$
\end{cor}

\begin{proof}
For simplicity of notation, we shall identify $F$ and $G$ with their images in $F\times G$. So suppose $\phi\in Z^1(F\times G,X)$. Then, by Theorem \ref{zero restriction}, $\phi|_F\in B^1(F,X)$ whereas $\phi|_G\in B^1(G,X)$. So, for some $x,y\in X$, we have $\phi(f)=x-\pi(f)x$ and $\phi(g)=y-\pi(g)y$ for all $f\in F$ and $g\in G$.

We claim that $x=y$. For otherwise, as $X^F=\{0\}$, there is some $f\in F$ so that $\big(I-\pi(f)\big)(y-x)\neq 0$ and therefore, as $X^G=\{0\}$, some $g\in G$ so that 
$$
\big(I-\pi(g)\big)\big(I-\pi(f)\big)(y-x)\neq 0.
$$
However, this contradicts that
\maths{
0
&= \phi(gf)-\phi(fg)\\
&=\pi(g) \phi(f)+\phi(g)-\pi(f)\phi(g)-\phi(f)\\
&=\pi(g)x-\pi(g)\pi(f)x+y-\pi(g)y-\pi(f)y+\pi(f)\pi(g)y-x+\pi(f)x\\
&=\big(I-\pi(g)\big)\big(I-\pi(f)\big)(y-x).
}

Thus, $x=y$ and so, for all $f\in F$ and $g\in G$, we have that $\phi(fg)=\pi(f)x-\pi(f)\pi(g)x+x-\pi(f)x=x-\pi(fg)x$, which shows that $\phi\in B^1(F\times G,X)$ and therefore that $H^1(F\times G,X)=\{0\}$. 
\end{proof}


\section{Homotopy of cochain maps}
Suppose $G$ and $F$ are groups and $X$ and $Y$ are $G$ and $F$-Banach modules respectively. A {\em cochain map (of degree $0$)} between the standard cochain complexes $C^\bullet(G,X)$ and $C^\bullet(F,Y)$ is a sequence $T=(T^n)_{n=0}^\infty$ of continuous linear operators 
$$
C^n(G,X)\overset{T^n}\longrightarrow C^n(F,Y)
$$
so that the following diagram commutes 
$$
\begin{tikzcd}
C^0(G,X)\arrow{d}{T^0}\arrow{r}{\partial^{1}}      & C^1(G,X)\arrow{d}{T^1}\arrow{r}{\partial^{2}}      & C^2(G,X)\arrow{d}{T^2}\arrow{r}{\partial^{3}}      &C^3(G,X)\arrow{d}{T^3}\arrow{r}{\partial^{4}}      & {\ldots} \\
C^0(F,Y)\arrow{r}{\partial^{1}}      & C^1(F,Y)\arrow{r}{\partial^{2}}      & C^2(F,Y)\arrow{r}{\partial^{3}}      &  C^3(F,Y)\arrow{r}{\partial^{4}}  &{\ldots} \\
\end{tikzcd}
$$
As usual, we will most often suppress the index $n$ and simply write $T$ for the individual operators $T^n$. Because a cochain map $T$ commutes with the boundary operator $\partial$ or rather as $T^{n+1}\partial^{n+1}=\partial^{n+1}T^n$, we observe that $T$ maps $Z^\bullet(G,X)$ into $Z^\bullet(F,Y)$ and $B^\bullet(G,X)$ into $B^\bullet(F,Y)$. It therefore follows that a cochain map $T$ defines a map on cohomology, i.e.,
$$
H^n(G,X)\overset{T}\longrightarrow H^n(F,Y).
$$

If $T$ and $S$ are two cochain maps from $C^\bullet(G,X)$ to $C^\bullet(F,Y)$, we say that $T$ and $S$ are {\em homotopic} provided that there is a sequence of continuous linear operators 
$$
C^{n+1}(G,X)\overset{R^{n+1}}\longrightarrow C^n(F,Y)
$$
so that, when letting $C^{0}(G,X)\overset{R^0}\longrightarrow \{0\}\overset{\partial^0}\longrightarrow C^0(F,Y)$ be the $0$ maps, we have that
$$
T^n-S^n=\partial^nR^{n}+R^{n+1}\partial^{n+1}
$$
for all $n\geqslant 0$. This gives us a non-commutative diagram as follows.
\vspace{.4cm}
$$
\begin{tikzcd}[row sep=large, column sep=large]   
&C^0(G,X)\arrow[d, shift left, "{T^0}" ', "{S^0}"]      \arrow{r}{\partial^{1}}   \arrow[dl, dashrightarrow, "{R^{0}}" description]
& C^1(G,X)\arrow[d, shift left, "{T^1}" ', "{S^1}"]   \arrow{r}{\partial^{2}}       \arrow[dl, dashrightarrow, "{R^{1}}" description]
& C^2(G,X)\arrow[d, shift left, "{T^2}" ', "{S^2}"]   \arrow{r}{\partial^{3}}       \arrow[dl, dashrightarrow, "{R^{2}}" description]
& {\ldots}  \arrow[dl, dashrightarrow, "{R^{2}}" description]  \\
\{0\}\arrow{r}{\partial^{0}}    
&C^0(F,Y)\arrow{r}{\partial^{1}}      
& C^1(F,Y)\arrow{r}{\partial^{2}}      
& C^2(F,Y)\arrow{r}{\partial^{3}}      
&{\ldots} \\
\end{tikzcd}
$$

The main observation regarding homotopic cochain maps $T$ and $S$ is that they define identical maps on cohomology, 
$$
H^n(G,X)\overset{T, S}\longrightarrow H^n(F,Y).
$$
Indeed, if $T^n-S^n=\partial^nR^{n}+R^{n+1}\partial^{n+1}$ and $\phi\in Z^n(G,X)={\sf ker}\, \partial^{n+1}$, then
$$ 
T^n\phi-S^n\phi=\partial^nR^{n}\phi+R^{n+1}\partial^{n+1}\phi=\partial^nR^{n}\phi\in B^n(F,Y).
$$

Our first lemma is related to the well-known fact that conjugation in groups is cohomologically trivial (see, e.g., Proposition 8.3 \cite{brown}).
\begin{prop}\label{homotopic cochain maps}
Suppose $G$ is a group, $F$ is a subgroup and $(X,\pi)$ is a Banach $G$-module. Assume that $\xi \in F' \cap \Delta G$ and define cochain maps
$$
C^\bullet(G,X)\overset{S, T}\longrightarrow C^\bullet(F,X)
$$
by $S \phi=\phi|_F$ and $T\phi=\pi(\xi)\circ \phi|_F$. Then $S$ and $T$ are homotopic.
\end{prop}

\begin{proof}
That the restriction map $S$ is a cochain map, i.e., commutes with $\partial$, is straightforward to see. That $T$ is a cochain map follows from the fact that $\pi(\xi)$ commutes with $\pi(f)$ for all $f\in F$. Define now $C^{n+1}(G,X)\overset{R^{n+1}}\longrightarrow C^n(F,X)$ for $n\geqslant 0$  by 
\maths{
(R^{n+1}\phi)(f_1,\ldots,f_n)=\sum_{i=1}^{n+1}(-1)^{i+1}\phi(f_1,\ldots, f_{i-1}, \xi, f_i,\ldots, f_n)
}
and let $R^{0}x=0$ for all $x\in X=C^0(G,X)$. 

Observe first that, for $x\in X=C^0(G,X)$, we have 
$$
Sx-Tx=x-\pi(\xi) x=R^1\partial^1x=R^1\partial^1x+\partial^0R^{0}x.
$$
We also claim that,  for all $\phi\in C^{n}(G,X)$ with $n\geqslant 1$ and $f_1,\ldots,f_n\in F$, we have
$$
\phi(f_1,\ldots,f_n)-\pi(\xi)\phi (f_1,\ldots,f_n)=(R^{n+1}\partial^{n+1}\phi)(f_1,\ldots,f_n)+ (\partial^{n}R^{n}\phi)(f_1,\ldots,f_n),
$$
whereby $S-T=R\partial+\partial R$ and thus showing that $S$ and $T$ are homotopic.

To see this, we compute
\maths{
(\partial^{n}R^{n}&\phi)(f_1,\ldots, f_n)\\
=&-\pi(f_1)(R^{n}\phi)(f_2,\ldots, f_n)
+(-1)^{n-1}(R^{n}\phi)(f_1,\ldots, f_{n-1})\\
&-\sum_{j=1}^{n-1}(-1)^j(R^{n}\phi)(f_1,\ldots, f_{j-1}, f_jf_{j+1}, f_{j+2}, \ldots, f_n)\\
%
%
=&-\pi(f_1)\Big[\sum_{i=1}^n(-1)^{i+1}\phi(f_2,\ldots, f_{i},\xi,f_{i+1}, \ldots, f_n)\Big]\\
&+   (-1)^{n-1}  \Big[\sum_{i=1}^n(-1)^{i+1}\phi(f_1,\ldots, f_{i-1},\xi,f_{i}, \ldots, f_{n-1})\Big] \\
&-\sum_{j=1}^{n-1}(-1)^j    \sum_{i=j+2}^{n+1}(-1)^{i}  \phi(f_1,\ldots,  f_{j-1}, f_jf_{j+1}, f_{j+2}, \ldots,f_{i-1},\xi, f_i, \ldots, f_n)   \\
&-\sum_{j=1}^{n-1}(-1)^j    \sum_{i=1}^{j}(-1)^{i+1}    \phi(f_1,\ldots, f_{i-1},\xi, f_{i},\ldots, f_{j-1}, f_jf_{j+1}, f_{j+2},  \ldots, f_n)   \\
%
%
=&-\pi(f_1)\Big[\sum_{i=2}^{n+1}(-1)^i\phi(f_2,\ldots, f_{i-1},\xi,f_{i}, \ldots, f_n)\Big]\\
&+   (-1)^n  \Big[\sum_{i=1}^n(-1)^i\phi(f_1,\ldots, f_{i-1},\xi,f_{i}, \ldots, f_{n-1})\Big] \\
&+\sum_{i=3}^{n+1}   \sum_{j=1}^{i-2}(-1)^{i+j-1}  \phi(f_1,\ldots,  f_{j-1}, f_jf_{j+1}, f_{j+2}, \ldots,f_{i-1},\xi, f_i, \ldots, f_n)   \\
&+\sum_{i=1}^{n-1}   \sum_{j=i}^{n-1}(-1)^{i+j}    \phi(f_1,\ldots, f_{i-1},\xi, f_{i},\ldots, f_{j-1}, f_jf_{j+1}, f_{j+2},  \ldots, f_n).  \\
}


Thus, 
\maths{
(R^{n+1}&\partial^{n+1}\phi)(f_1,\ldots, f_n)
=\sum_{i=1}^{n+1}(-1)^{i+1}(\partial^{n+1}\phi)(f_1,\ldots, f_{i-1}, \xi, f_i,\ldots, f_n)\\
%
%
=&-\pi(\xi)\phi(f_1,\ldots, f_n)
-\sum_{i=2}^{n+1}(-1)^{i+1}\pi(f_1)\phi(f_2,\ldots, f_{i-1}, \xi, f_{i},\ldots, f_n)\\
&+\sum_{i=1}^{n}(-1)^{i+1+n}\phi(f_1,\ldots, f_{i-1}, \xi, f_{i},\ldots, f_{n-1})+(-1)^{2n+2}\phi(f_1,\ldots, f_n)\\
&-\sum_{i=3}^{n+1}\sum_{j=1}^{i-2}(-1)^{i+1+j}\phi(f_1,\ldots, f_{j-1}, f_jf_{j+1}, f_{j+2}, \ldots, f_{i-1},\xi, f_{i}, \ldots, f_n)   \\
&-\sum_{i=1}^{n-1}\sum_{j=i}^{n-1}(-1)^{i+1+j+1}\phi(f_1,\ldots, f_{i-1},\xi, f_i, \ldots, f_{j-1}, f_jf_{j+1}, f_{j+2}, \ldots, f_n)   \\
&-\sum_{i=1}^n(-1)^{i+1+i}\phi(f_1,\ldots,  f_{i-1}, \xi f_i, f_{i+1},\ldots, f_n)\\
&-\sum_{i=2}^{n+1}(-1)^{i+1+i-1}\phi(f_1,\ldots,  f_{i-2}, f_{i-1}\xi, f_{i},\ldots, f_n)\\
%
%
=&\big(I-\pi(\xi)\big)\phi(f_1,\ldots, f_n)
+\pi(f_1)\Big[\sum_{i=2}^{n+1}(-1)^i\phi(f_2,\ldots, f_{i-1}, \xi, f_{i},\ldots, f_n)\Big]\\
&-(-1)^{n}\sum_{i=1}^{n}(-1)^{i}\phi(f_1,\ldots, f_{i-1}, \xi, f_{i},\ldots, f_{n-1})\\
&-\sum_{i=3}^{n+1}\sum_{j=1}^{i-2}(-1)^{i+j-1}\phi(f_1,\ldots, f_{j-1}, f_jf_{j+1}, f_{j+2}, \ldots, f_{i-1},\xi, f_{i}, \ldots, f_n)   \\
&-\sum_{i=1}^{n-1}\sum_{j=i}^{n-1}(-1)^{i+j}\phi(f_1,\ldots, f_{i-1},\xi, f_i, \ldots, f_{j-1}, f_jf_{j+1}, f_{j+2}, \ldots, f_n)   \\
&+\sum_{i=1}^n\phi(f_1,\ldots,  f_{i-1}, \xi f_i, f_{i+1},\ldots, f_n)\\
&-\sum_{i=2}^{n+1}\phi(f_1,\ldots,  f_{i-2}, \xi f_{i-1}, f_{i+1},\ldots, f_n)\\
%
%
=&\big(I-\pi(\xi)\big)\phi(f_1,\ldots, f_n)-(\partial^{n}R^{n}\phi)(f_1,\ldots, f_n).
}
So $S-T=R\partial+\partial R$ as claimed.
\end{proof}

The next lemma is well-known in  the discrete setting (see, e.g., Proposition 0.3 \cite{brown}). For the sake of completeness, we include a proof adapted to our setup. 
\begin{lemme}\label{split exact}
The following conditions are equivalent for a group $G$ and a Banach $G$-module $(X,\pi)$.
\begin{enumerate}
\item The cochain identity map $C^\bullet(G,X)\overset I\longrightarrow C^\bullet(G,X)$ is null-homotopic,
\item there exists an invertible null-homotopic cochain map $C^\bullet(G,X)\longrightarrow C^\bullet(G,X)$,
\item the cochain complex
$$
\{0\}\overset{\partial^0}\longrightarrow C^0(G,X)\overset{\partial^1}\longrightarrow  C^1(G,X)\overset{\partial^2} \longrightarrow C^2(G,X)\overset{\partial^3} \longrightarrow C^3(G,X)\overset{\partial^4} \longrightarrow \cdots
$$
is {\em split exact}, that is, $H^n(G,X)=\{0\}$  and $B^n(G,X)=Z^n(G,X)$ is complemented in $C^n(G,X)$ for all $n\geqslant 0$.
\end{enumerate}
\end{lemme}

Before beginning the proof, observe that, when the cochain complex is split exact, say 
$$
C^n(G,X)=B^n(G,X)\oplus Y^n
$$ 
for some sequence $Y^0, Y^1, \ldots$ of closed linear subspaces $Y^n\subseteq C^n(G,X)$, then we get exact sequences
\maths{
0\longrightarrow Y^{n-1}\overset{\partial^n}\longrightarrow B^n(G,X)\oplus Y^n \overset{\partial^{n+1}}\longrightarrow B^{n+1}(G,X)\longrightarrow 0
}
where $B^n(G,X)={\sf rg}\,\partial^n={\sf ker}\,\partial^{n+1}$ and so $ Y^{n-1} \overset{\partial^{n}}\longrightarrow B^{n}(G,X)$ is an isomorphism.

\begin{proof}
(2)$\saa$(1):  Suppose that  $C^\bullet(G,X)\overset T\longrightarrow C^\bullet(G,X)$ is an invertible null-homotopic cochain map. This means that we can write
$$
T=\partial R+R\partial
$$
for some continuous linear operators $C^{\bullet+1}(G,X)\overset R\longrightarrow C^\bullet(G,X)$. Now, as $T$ commutes with $\partial$, so does $T\inv$, whereby
$$
I=T\inv T=T\inv \partial R+T\inv  R\partial= \partial \circ (T\inv R)+(T\inv R)\circ \partial
$$
and so $I$ is null-homotopic too.

(1)$\saa$(3): If (1) holds, this means that we may write $I=\partial R+R\partial$ for some continuous linear operators $C^{\bullet}(G,X)\overset R\longrightarrow C^{\bullet-1}(G,X)$, where we set $C^{-1}(G,X)=\{0\}$. In that case,  the composition
$$
C^\bullet(G,X)\overset {\partial R}\longrightarrow C^\bullet(G,X)
$$
satisfies
$$
(\partial R)^2=\partial R\partial R+R0R=\partial R\partial R+R\partial \partial R=I\partial R=\partial R,
$$
that is, $\partial R$ is a continuous linear projection onto its image. Also, if $\psi\in B^n(G,X)$, write $\psi=\partial \phi$ for some $\phi\in C^{n-1}(G,X)$ and note that
$$
\partial R\psi=\partial R\partial \phi=\partial R\partial \phi+R\partial \partial \phi=I\partial \phi=\psi.
$$
Therefore, ${\sf rg}\, (\partial R)=B^\bullet(G,X)$ and $\partial R$ is a linear projection of $C^\bullet(G,X)$ onto the closed linear subspace $B^\bullet(G,X)$ with complementary projection $I-\partial R=R\partial$. Since $Z^\bullet(G,X)\subseteq {\sf ker}\, R\partial$, it follows that also $Z^\bullet(G,X)=B^\bullet(G,X)$. In other words,  $H^\bullet(G,X)=\{0\}$ and the cochain complex is split exact.

(3)$\saa$(2): Suppose conversely that the cochain complex is split exact. Then, for all $n\geqslant 0$, we have $Z^n(G,X)=B^n(G,X)$ and 
$$
C^n(G,X)=B^n(G,X)\oplus Y^n
$$
for some closed linear subspace $Y^n\subseteq C^n(G,X)$. 
Furthermore, $Y^n\overset{\partial^{n+1} }\longrightarrow B^{n+1}(G,X)$ is an isomorphism and thus admits an inverse. So let $C^{n+1}(G,X)\overset {R^{n+1}}\longrightarrow Y^n$ be $0$ on $Y^{n+1}$ and the inverse of $\partial^{n+1}$ on $B^{n+1}(G,X)$. Also, $R^0$ is the unique operator $C^{0}(G,X)\overset {R^{0}}\longrightarrow \{0\}$.
Then $\partial^{n}R^{n}$ is the continuous linear projection of $C^{n}(G,X)$ onto $B^{n}(G,X)$ along $Y^{n}$, whereas $R^{n+1}\partial^{n+1}$ is the complementary projection onto $Y^{n}$. That is, $I=\partial^{n}R^n+R^{n+1}\partial^{n+1}$, showing that the invertible operator $I$ is null-homotopic.
\end{proof}

\begin{thm}\label{split exact thm}
Suppose $G$ is a group  and $X$ a Banach $G$-module. Assume that $\xi\in G'\cap \Delta G$ is chosen so that $I-\pi(\xi)$ is invertible.  Then the cochain complex
$$
\{0\}\overset{\partial^0}\longrightarrow C^0(G,X)\overset{\partial^1}\longrightarrow  C^1(G,X)\overset{\partial^2} \longrightarrow C^2(G,X)\overset{\partial^3} \longrightarrow C^3(G,X)\overset{\partial^4} \longrightarrow \cdots
$$
is split exact.
\end{thm}

\begin{proof}
As $\xi\in G'\cap \Delta G$, we may apply Proposition \ref{homotopic cochain maps} with $F=G$ to conclude that the cochain map
$$
\phi\mapsto \big(I-\pi(\xi)\big)\circ \phi
$$
is null-homotopic. Also, because the operator ${I-\pi(\xi)}$ is invertible, so is the above cochain map. Therefore,  by Lemma \ref{split exact}, the cochain complex 
$$
\{0\}\overset{\partial^0}\longrightarrow C^0(G,X)\overset{\partial^1}\longrightarrow  C^1(G,X)\overset{\partial^2} \longrightarrow C^2(G,X)\overset{\partial^3} \longrightarrow C^3(G,X)\overset{\partial^4} \longrightarrow \cdots
$$
is split exact.
\end{proof}

\begin{cor}\label{split exact FC centre}
Suppose $G$ is a group and $(X,\pi)$ is a uniformly convex isometric Banach $G$-module. Let $\Phi$ denote the FC-centre of $G$ and assume that $X$ has no almost invariant unit vectors as a $\Phi$-module. Then the cochain complex
$$
\{0\}\overset{\partial^0}\longrightarrow C^0(G,X)\overset{\partial^1}\longrightarrow  C^1(G,X)\overset{\partial^2} \longrightarrow C^2(G,X)\overset{\partial^3} \longrightarrow C^3(G,X)\overset{\partial^4} \longrightarrow \cdots
$$
is split exact.
\end{cor}

\begin{proof}
Because $X$ is uniformly convex and has no almost invariant unit vectors as a $\Phi$-module, by Lemma \ref{invertibility}, there is a finite set $E\subseteq \Phi$ so that $\norm{\pi(\xi)}<1$ for all $\xi \in \Delta G$ with $E\subseteq {\sf supp}\,(\xi)$. Enlarging $E$, we may suppose that it is a union of conjugacy classes in $G$, whereby $\xi= \frac1{|E|}\sum_{g\in E}g\in G'\cap \Delta G$. By Neumann's Lemma, $I-\pi(\xi)$ is invertible. The result thus follows directly from Theorem \ref{split exact thm}.
\end{proof}

\begin{thm}
Suppose $G$ is a group and $(X,\pi)$ is a uniformly convex isometric Banach $G$-module. Assume also that $F$ is a subgroup of $G$ so that $X$ has no almost invariant unit vectors as a ${\sf FC}_G(F)$-module. Then the restriction map
$$
H^n(G,X)\to H^n(F,X)
$$
is zero for all $n\geqslant 0$.
\end{thm}

\begin{proof}
Because $X$ is uniformly convex and has no almost invariant unit vectors as a ${\sf FC}_G(F)$-module, by Lemma \ref{invertibility}, there is a finite set $E\subseteq {\sf FC}_G(F)$ so that $\norm{\pi(\xi)}<1$ for all $\xi \in \Delta\big( {\sf FC}_G(F)\big)$ with $E\subseteq {\sf supp}\,(\xi)$. Enlarging $E$, we may suppose that it is a union of $F$-conjugacy classes, whereby $\xi= \frac1{|E|}\sum_{g\in E}g\in F'\cap \Delta G$. By Neumann's Lemma, $I-\pi(\xi)$ is invertible. 
Observe now that the restriction map $C^n(G,X)\to C^n(F,X)$ is the composition of the two cochain maps
$$
C^n(G,X)\overset\Theta\longrightarrow C^n(F,X)\overset\Omega\longrightarrow C^n(F,X)
$$
defined by
$$
\Theta(\phi)=(I-\pi(\xi))\circ \phi|_F, \qquad \Omega(\psi)=(I-\pi(\xi))\inv\circ \psi.
$$
By Proposition \ref{homotopic cochain maps}, $\Theta$ is null-homotopic and hence so is the composition $\Omega\Theta$.
\end{proof}


\section{Reduced cohomology}\label{reduced cohomology}

Recall that, when $G$ is a group and $(X,\pi)$ is a Banach $G$-module, the space $C^n(G,X)$ of $n$-cochains is nothing but the product 
$$
\prod_{G^n}X
$$
and thus is naturally equipped with the Tychonoff product topology. Because $C^n(G,X)\overset{\partial^{n+1}}\longrightarrow C^{n+1}(G,X)$ is a continuous linear operator, the kernel $Z^n(G,X)={\sf ker}\,\partial^{n+1}$ will be closed in $C^n(G,X)$. In contradistinction, $B^n(G,X)={\sf rg}\,\partial^n$ may not be closed and thus the cohomology $H^n(G,X)=Z^n(G,X)/B^n(G,X)$ may not be a Hausdorff space. For this reason, we define the  {\em reduced cohomology} of the Banach module to be the quotient space
$$
\ov H^n(G,X)=C^n(G,X)\Big/\ov{B^n(G,X)}.
$$

\begin{lemme}\label{preryll}
Suppose $X$ is a strictly convex, reflexive Banach space and that $\ku S\subseteq \ku L(X)$ is a semigroup of contractions, that is, so that $\norm{S}\leqslant 1$ for all $S\in \ku S$. Assume furthermore that $\bigcap_{\ku S}\ker(I-S)=\{0\}$. Then, for all finite subsets $E\subseteq X$ and $\eps>0$, there is some $T\in {\sf conv}(\ku S)$ so that
$$
sup_{x\in E}\norm{Tx}<\eps.
$$
\end{lemme}

\begin{proof}
Suppose first that a single vector $x\in X$ and $\eps>0$ are given. We note that $K=\ov{\sf conv}(\ku S x)$ is a closed convex $\ku S$-invariant subset of $X$. Therefore, as $X$ is strictly convex and reflexive, $K$ has a unique element $z\in K$ of minimal norm. As $\norm{Sz}\leqslant \norm z$  and $Sz\in K$ for all $S\in \ku S$, it follows that $z\in \bigcap_{\ku S}\ker(I-S)=\{0\}$. We may thus find some $T\in {\sf conv}(\ku S)$ so that $Tx\in {\sf conv}(\ku S x)$ has norm less than $\eps$.

Now, if instead a finite set $E=\{x_1,\ldots,x_n\}\subseteq X$ and $\eps>0$ are given, we pick successively $T_1,\ldots, T_n\in {\sf conv}(\ku S)$ so that $\norm{T_iT_{i-1}\cdots T_1x_i}<\eps$ and let $T=T_n\cdots T_1\in {\sf conv}(\ku S)$. It then follows that
$$
\norm{Tx_i}\leqslant \norm{T_n\cdots T_{i+1}}\cdot\norm{T_iT_{i-1}\cdots T_1x_i}<\eps
$$
for all $i$.
 \end{proof}

In the context of reduced cohomology, the following result takes the place of Lemma \ref{invertibility}.

\begin{lemme}\label{ryll}
Suppose $G$ is a group and $(X,\pi)$ is a separable reflexive isometric Banach $G$-module. Assume also that $F$ is a subgroup of $G$ so that $X^{{\sf FC}_G(F)}=\{0\}$. Then, for all finite subsets $E\subseteq X$ and $\eps>0$, there is some $\xi\in F'\cap \Delta G$ so that 
$$
\norm{\pi(\xi)x}<\eps
$$
for all $x\in E$.
\end{lemme}

\begin{proof}
Recall first that, by the previously mentioned theorem of G. Lancien \cite{lancien}, because $(X, \norm\cdot)$ is separable reflexive, it admits an isometry-invariant locally uniformly convex renorming $(X,\triple\cdot)$ and thus so that the module action $G\overset \pi\curvearrowright (X, \triple\cdot)$ remains isometric. We may thus simply assume that  $X$ is strictly convex and reflexive. 

Consider the convex semigroup of contractions 
$$
\ku S=\{\pi(\xi)\in \ku L(X)\del \xi\in F'\cap \Delta G\}
$$ 
and observe, by Lemma \ref{X^F'},  that $\bigcap_{S\in \ku S}\ker(I-S)=\{0\}$. Therefore, by Lemma \ref{preryll}, we find that, for all finite subsets $E\subseteq X$ and $\eps>0$, there is some $\xi\in  F'\cap \Delta G$ so that
$$
\norm{\pi(\xi)x}<\eps
$$
for all $x\in E$.
\end{proof}

We are now able to obtain a strengthening of the central result of \cite{BRS} for the case of separable reflexive isometric Banach modules.
\begin{thm}\label{ov H}
Suppose that $G$ is a group and $(X,\pi)$  a separable reflexive isometric Banach $G$-module. Assume also that $F\leqslant G$ is a subgroup with $X^{{\sf FC}_G(F)}=\{0\}$. Then the restriction map
$$
\ov H^n(G,X)\to \ov H^n(F,X)
$$
is zero for all $n\geqslant 0$.
\end{thm}

\begin{proof}
Suppose that $\phi\in Z^n(G,X)$. We must show that $\phi|_F\in \ov{B^n(F,X)}$. That is, we must verify that for all finite subsets $E\subseteq F^n$ and $\eps>0$ there is some $\psi\in B^n(F,X)$ so that
$$
\norm{\phi(\vec f)-\psi(\vec f)}<\eps
$$
for all $\vec f\in E$. So let $E$ and $\eps$ be given and pick by Lemma \ref{ryll} some $\xi\in F'\cap \Delta G$ so that 
$$
\norm{\pi(\xi)\phi(\vec f)}<\eps
$$
for all $\vec f\in E$. Then
$$
\Norm{\phi(\vec f)-\big(\phi-\pi(\xi)\circ \phi\big)(\vec f)}=\norm{\pi(\xi)\phi(\vec f)}<\eps
$$
whereas
$$
\big(\phi-\pi(\xi)\circ \phi\big)\big|_F\in B^n(F,X)
$$
by Proposition \ref{homotopic cochain maps}. So $\psi=\big(\phi-\pi(\xi)\circ \phi\big)\big|_F$ is as required.
\end{proof}

\begin{prop}
Suppose $G$ is a group with a finite generating set $\Sigma$ and $(X,\pi)$ is a separable reflexive isometric Banach $G$-module with $X^G=\{0\}$. Then, for all $\phi\in Z^1(G,X)$  and  $\eps>0$, there is some $x\in X$ with
$$
\NORM{\sum_{g\in \Sigma}\big(\phi-\partial x\big)(g)}<\eps.
$$
\end{prop}

\begin{proof} 
As in the proof of Lemma \ref{ryll}, we may assume that $X$ is strictly convex. Also, as $(\phi-\partial x)(1)=0$ for all $x\in X$, we can assume that $1\in \Sigma$.
Therefore, if $\sigma=\frac1{|\Sigma|}\sum_{g\in \Sigma}g$, the operator $\pi(\sigma)$ has no non-zero invariant vectors. 

Set also
$$
\ku S=\big\{  \pi(\xi)\in \ku L(X)
\del \xi \in {\sf conv}
\{\sigma^n\,|\, n\geqslant 1\}
\big\}.
$$
Then $\ku S$ is a convex semigroup of contractions with $\bigcap_{\ku S}\ker(I-S)=\{0\}$ and so, by Lemma \ref{preryll}, we may find  some $\xi\in  \ku S$ so that
$$
\norm{\pi(\xi)\phi(\sigma)}<\frac\eps {|\Sigma|}.
$$
However, $\sigma\xi=\xi\sigma$ and thus for $x=\phi(\xi)$
$$
\partial x(\sigma)=\big(I-\pi(\sigma)\big)\phi(\xi)=\big(I-\pi(\xi)\big)\phi(\sigma),
$$
whereby
$$
\NORM{\sum_{g\in \Sigma}\big(\phi-\partial x\big)(g)}
={|\Sigma|}\cdot \Norm{\phi(\sigma)-\partial x(\sigma)}
={|\Sigma|}\cdot \Norm{\pi(\xi)\phi(\sigma)}<\eps.
$$
\end{proof}


\section{Affine actions}\label{affine actions}
Recall that, by the Mazur--Ulam Theorem, every surjective isometry $X\overset A\longrightarrow X$ of a Banach space is affine, that is, satisfies 
$$
A\Big(\sum_{i=1}^nt_ix_i\Big)=\sum_{i=1}^nt_iA(x_i)
$$
for all $t_i\in \R$ and $x_i\in X$ so that $\sum_{i=1}^nt_i=1$. Also, if $X\overset A\longrightarrow X$ is any continuous affine map, there is a unique bounded linear operator $X\overset T\longrightarrow X$ and a vector $a\in X$ so that
$$
Ax=Tx+a
$$
for all $x\in X$. Furthermore, if $B$ is another affine map given by $Bx=Sx+b$ for some operator $S$ and vector $b$, then one sees that
$$
ABx=TSx+(Tb+a).
$$
Therefore, if  $G\overset\alpha\curvearrowright X$ is a continuous affine action by a group $G$, we obtain a continuous linear action $G\overset\pi\curvearrowright X$  along with a  map $G\overset \phi\longrightarrow X$ so that
$$
\alpha(g)x=\pi(g)x+ \phi(g)
$$
and 
$$
\phi(gf)=\pi(g)\phi(f)+\phi(g)
$$
for all $g,f\in G$ and $x\in X$. In other words, $\phi$ is nothing but a $1$-cocycle associated with the Banach $G$-module $(X,\pi)$. 

Observe also that, as
$$
\partial x(g)-\phi(g)=x-\pi(g)x-\phi(g)=x-\alpha(g)x,
$$
we have that $\phi=\partial x$ if and only if $x$ is fixed by the action  $G\overset\alpha\curvearrowright X$.

Similarly, the cocycle $\phi$ belongs to $\ov{ B^1(G,X)}$ if and only if, for every finite subset $E\subseteq G$ and $\eps>0$, there is an $x\in X$ so that
$$
\sup_{g\in E}\norm{x-\alpha(g)x}=\sup_{g\in E}\norm{\partial x(g)-\phi(g)}<\eps.
$$

These computations show that, for a fixed Banach $G$-module $(X,\pi)$, there is a bijective correspondence 
$$  
\{\text{Affine actions with linear part $\pi$}\}  \leftrightsquigarrow Z^1(G,X)
$$
between the collection of affine actions $G\overset\alpha\curvearrowright X$ whose linear part is $\pi$ and then the space  $Z^1(G,X)$ of $1$-cocycles $\phi$ given by $\phi(g)=\alpha(g)0$. Furthermore, under this correspondence, the space $B^1(G,X)$ of coboundaries correspond exactly to actions fixing a point in $X$, whereas the space $\ov{B^1(G,X)}$ of {\em almost coboundaries} corresponds to actions having almost fixed points in the following sense.

\begin{defi}
An affine isometric group action $G\overset \alpha\curvearrowright X$ on a Banach space {\em almost fixes a point} if and only if, for all finite subsets $E\subseteq G$ and $\eps>0$, there is some $x\in X$ so that 
$$
\norm{x-\alpha(g)x}<\eps
$$
for all $g\in E$.
\end{defi}

Suppose now that $(X,\pi)$ is an isometric Banach $G$-module and $\phi\in Z^1(G,X)$ is a cocycle with associated affine isometric action $G\overset\alpha\curvearrowright X$. Then because
$$
{\sf im}\, \phi=\{\alpha(g)0\del g\in G\}
$$
we find that the cocycle $\phi$ is bounded if and only if the orbit of $0$ and therefore every orbit is bounded. In this case, $\ov{\sf conv}\big( {\sf im}\, \phi\big)$ is a norm-bounded closed convex $\alpha(G)$-invariant set. 
In particular, if $X$ is reflexive, this implies by the Ryll-Nardzewski fixed-point theorem \cite{ryll} (see also \cite{namioka}) that there is a fixed point for  $G\overset\alpha\curvearrowright X$. In other words, when $X$ is reflexive, $1$-coboundaries are simply the bounded $1$-cocycles.

Consider now a continuous affine group action $G\overset\alpha\curvearrowright X$ and let $\pi$ and $\phi$ denote the associated linear part, respectively, associated cocycle. As before, we canonically extend these to an algebra representation $\R G\overset \pi\longrightarrow \ku L(X)$, respectively, an affine map $\A G\overset \phi\longrightarrow X$. By Remark \ref{no worries}, we have $\phi(\xi\zeta)=\pi(\xi)\phi(\zeta)+\phi(\xi)$ for all $\xi,\zeta\in \A G$ and therefore $\alpha$ extends to an action
$$
\A G\overset\alpha\curvearrowright X
$$
of the multiplicative semigroup $\A G$ by continuous affine transformations of $X$ by setting 
$$
\alpha(\xi)x=\pi(\xi)x+\phi(\xi)
$$
for all $\xi \in \A G$ and $x\in X$. 

Observe then that, because both $\pi$ and $\phi$ are affine mappings, we have that
$$
\alpha\Big(\sum_{i=1}^nt_i\xi_i\Big)(x)= \sum_{i=1}^nt_i\alpha(\xi_i)x
$$
whenever $x\in X$, $\xi_i\in \A G$, $t_i\in \R$ and $\sum_{i=1}^nt_i=1$. In particular, this implies that the convex hull of the orbit $\alpha(G)x$ of $x$ under the affine action by $G$ can be written as
$$
{\sf conv}\big(\alpha(G)x\big)={\alpha\big(\Delta G\big)x}.
$$
Note that another way of expressing these facts is to say that the action mapping 
$$
\A G\times X\overset \alpha\longrightarrow X
$$
is biaffine.


\section{Appendix: Quotients, complementation and property (T)}

Thus far, we have freely used the assumption that $X^{{\sf FC}_G(F)}=\{0\}$ for some subgroup $F$ of $G$ and it is therefore useful to pause to consider the import of this hypothesis. A good deal of the material presented here is well-known and implicit in the literature in various forms. In particular, this applies to the case of isometric reflexive Banach $G$-modules, where some of the arguments below can be significantly simplified by using the Alaoglu--Birkhoff decomposition $X=X^G\oplus X_G$. It may be of value however to observe that reflexivity or weak almost periodicity is not essential for the results that follow.

\begin{exa}[Decompositions along ${\sf FC}_G(F)$ for reflexive Banach modules]
Assume first that $(X,\pi)$ is a reflexive isometric Banach $G$-module and $F\leqslant G$ a subgroup so that $G=F\cdot {\sf FC}_G(F)$. Observe that, because ${\sf FC}_G(F)$ is normalised by $F$ and $G=F\cdot {\sf FC}_G(F)$, we have that ${\sf FC}_G(F)$ is normal in $G$ and therefore that the Alaoglu--Birkhoff decomposition
$$
X=X^{{\sf FC}_G(F)}\oplus X_{{\sf FC}_G(F)}
$$ 
is $G$-invariant. From this we obtain similar decompositions of $C^n(G,X)$,  $Z^n(G,X)$ and $B^n(G,X)$, which in turn induce a decomposition
$$
H^n(G,X)=H^n(G,X^{{\sf FC}_G(F)})\;\oplus\; H^n(G,X_{{\sf FC}_G(F)}).
$$
Working with the second summand $H^n(G,X_{{\sf FC}_G(F)})$ thus essentially corresponds to an initial assumption that  $X^{{\sf FC}_G(F)}=\{0\}$. 
\end{exa}

We now focus on the general case of isometric Banach modules where Alaoglu--Birkhoff decompositions may no longer be available.
So, suppose $(X,\pi)$ is a general  isometric Banach $G$-module and let $\ov{\cdot}\colon X\to X/X^G$ be the natural quotient map, that is, $\ov x=x+X^G$. Then we obtain an isometric Banach $G$-module $(X/X^G,\ov\pi)$ by letting
$$
\ov\pi(g)\ov x=\ov{\pi(g)x}
$$
for $x\in X$ and $g\in G$. With this definition, $\ov{\cdot}\colon X\to X/X^G$ is a $G$-equivariant map and we therefore obtain a cochain map
$$
\ov{\cdot}\colon C^\bullet (G,X)\to C^\bullet(G,X/X^G)
$$ 
by setting
$$
\ov\phi(g_1,\ldots, g_n)=\ov{\phi(g_1,\ldots, g_n)}
$$
for $\phi \in C^n(G,X)$.

\begin{lemme}\label{quotient displacement}
Suppose $G$ is a group and $(X,\pi)$ is an isometric Banach $G$-module. Let $g\in G$ and $x\in X$. Then
$$
\frac 12\cdot \norm{x-\pi(g)x}\;\leqslant\; \norm{\ov x-\ov{\pi}(g)\ov x}_{X/X^G}\;\leqslant\; \norm{x-\pi(g)x}.
$$
Also,  if $\phi\in Z^1(G,X)$ satisfies $\lim_n\frac{\phi(g^n)}n=0$, then
$$
\frac 12\cdot \norm{\phi(g)}\;\leqslant\; \norm{\ov{\phi}(g)}_{X/X^G}\;\leqslant\; \norm{\phi(g)}.
$$
\end{lemme}

\begin{proof}
Let $g\in G$, $x\in X$ and $\phi\in Z^1(G,X)$ be given and set $\delta_n=\frac1n\sum_{i=0}^{n-1}g^i\in \Delta G$. Then
\maths{
\pi(\delta_n)\big(x-\pi(g)x\big)
=\frac1n\sum_{i=0}^{n-1}\pi(g)^i\big(x-\pi(g)x\big)
=\frac{x-\pi(g)^nx}n
\;\Lim{n}0,
}
whereas
\maths{\pi(\delta_n)\phi(g)
=&\frac1n\big(\pi(g^{n-1})\phi(g)+\cdots+\pi(g^3)\phi(g)+\pi(g^2)\phi(g)+\pi(g)\phi(g)+\phi(g)\big)\\
=&\frac1n\big(\pi(g^{n-1})\phi(g)+\cdots+\pi(g^3)\phi(g)+\pi(g^2)\phi(g)+\phi(g^2)\big)\\
=&\frac1n\big(\pi(g^{n-1})\phi(g)+\cdots+\pi(g^3)\phi(g)+\phi(g^3)\big)\\
=&\ldots\\
=&\frac{\phi(g^n)}n.
}

To prove the result, it thus suffices to show that $\norm{y}\leqslant 2 \norm{\ov y}_{X/X^G}$ for all $y\in X$ satisfying $0\in \ov{\pi(\Delta G)y}$. To see this, suppose such a $y$ is given and find $\beta_n\in \Delta G$ so that $\pi(\beta_n)y\to 0$. Then, for all $z\in X^G$, 
\maths{
\norm{z}
=\lim_n\Norm{z+\pi(\beta_n)y}
=\lim_n\Norm{\pi(\beta_n)\big(z+y\big)}
\leqslant \norm{z+y}
}
and so 
\maths{
\norm{y}
\leqslant \norm{z+y}+\norm{z}
\leqslant 2\norm{z+y}.
}
This shows that $\norm{y}\,\leqslant\, 2 \cdot \inf_{z\in X^G}\norm{z+y}=2\norm{\ov y}_{X/X^G}$.
\end{proof}

\begin{lemme}\label{rem}
Suppose $G$ is a group and $(X,\pi)$ is an isometric Banach $G$-module. Then the cochain map $\phi\mapsto \ov \phi$ restricts to a topological isomorphism between the topological vector spaces $\ov{B^1(G,X)}$ and $\ov{B^1(G,X/X^G)}$ and also maps ${B^1(G,X)}$ bijectively onto  ${B^1(G,X/X^G)}$
\end{lemme}

\begin{proof}
Because $\ov{\partial^1 x}=\ov x-\ov\pi(\cdot)\ov x$, Lemma \ref{quotient displacement} implies that the continuous linear map $\phi\mapsto \ov \phi$ defines a uniform homeomorphism between the topological vector spaces  ${B^1(G,X)}$ and  ${B^1(G,X/X^G)}$. It therefore follows that it extends to a topological isomorphism between $\ov{B^1(G,X)}$ and $\ov{B^1(G,X/X^G)}$. \end{proof}

Observe that Lemma \ref{rem} in particular implies that $B^1(G,X)$ is closed in $Z^1(G,X)$ if and only if $B^1(G,X/X^G)$ is closed in $Z^1(G,X/X^G)$.

\begin{prop}
Suppose $G$ is a group with no non-trivial homomorphisms to $\R$ and $(X,\pi)$ is an isometric Banach $G$-module. 
\begin{enumerate}
\item If $H^1(G,X/X^G)=\{0\}$, then also $H^1(G,X)=\{0\}$.
\item If $\ov H^1(G,X/X^G)=\{0\}$, then also $\ov H^1(G,X)=\{0\}$.
\end{enumerate}
\end{prop}

\begin{proof}
By Lemma \ref{rem}, it suffices to show that $\phi\mapsto \ov \phi$ is injective on $Z^1(G,X)$.
So, assume $\phi\in Z^1(G,X)$ and $\ov \phi=0$. Then $\phi[G]\subseteq X^G$  and so
$$
\phi(gf)=\pi(g)\phi(f)+\phi(g)=\phi(g)+\phi(f)
$$
for all $g,f\in G$. In other words, such a $\phi$ is a homomorphism from $G$ to the additive group of $X$ and thus, composing with a linear functional, one obtains a homomorphism to $\R$. By the assumption on $G$, it follows that $\ov \phi=0$ implies that $\phi=0$ for all $\phi\in Z^1(G,X)$. In other words, the mapping $\phi\mapsto \ov \phi$ is injective on $Z^1(G,X)$.
\end{proof}

The next result applies, in particular, to f.g. groups generated by a set of elements of finite order such as the infinite dihedral group $D_\infty$.
\begin{prop}
Suppose $G$ is a group with a symmetric finite generating set $\Sigma$ and  associated length function $\ell$. Assume also that $\lim_n\frac {\ell(g^n)}n=0$ for all $g\in \Sigma$ and that $(X,\pi)$ is an isometric Banach $G$-module. Then the map $\phi\mapsto \ov \phi$ defines an isomorphic embedding of $\ov H^1(G,X)$ into $\ov H^1(G,X/X^G)$.
\end{prop}

\begin{proof}
Recall that $\ell$ is defined by  $\ell(g)=\min(k\del g=f_1\cdots f_k, f_i \in \Sigma)$. Also, for $\phi\in Z^1(G,X)$ and $f_i\in G$,
\maths{
\norm{\phi(f_1\cdots f_n)}
&=\Norm{\phi(f_1)+\pi(f_1)\phi(f_2)+\pi(f_1f_2)\phi(f_3)+\cdots+\pi(f_1\cdots f_{n-1})\phi(f_n)}\\
&\leqslant \sum_{i=1}^n\norm{\phi(f_i)}.
}
It thus follows that $\norm{\phi(g)}\leqslant \max_{f\in \Sigma}\norm{\phi(f)}\cdot \ell(g)$ for all $g\in G$ and, in particular, that
$$
\lim_n\frac{\norm{\phi(g^n)}}n\leqslant \max_{f\in \Sigma}\norm{\phi(f)}\cdot \lim_n\frac{\ell(g^n)}n=0
$$
for all $g\in \Sigma$. Applying Lemma \ref{quotient displacement}, we find that, for $g\in \Sigma$,
$$
\frac 12\cdot \norm{\phi(g)}\;\leqslant\; \norm{\ov{\phi}(g)}_{X/X^G}\;\leqslant\; \norm{\phi(g)}.
$$
Now, if $\lim_i\ov \phi_i=0$ for some sequence $\phi_i\in Z^1(G,X)$, then, for all $g\in \Sigma$, we have  $\lim_i\ov{\phi_i}(g)=0$  and thus $\lim_i\phi_i(g)=0$, whereby $\lim_i \phi_i=0$. In other words, the Fr\'echet space $Z^1(G,X)$ isomorphically embeds into the Fr\'echet space  $Z^1(G,X/X^G)$ via $\phi\mapsto \ov \phi$. As $\ov {B^1(G,X)}$ is mapped onto  $\ov {B^1(G,X/X^G)}$, we find that $\phi\mapsto \ov \phi$ induces an isomorphic embedding of the quotient  $\ov H^1(G,X)$ into $\ov H^1(G,X/X^G)$.
\end{proof}

The next lemma is now a simple adaptation of the corresponding result for Hilbert spaces due to A. Guichardet \cite{guichardet}.
\begin{lemme}\label{guichardet}
Suppose $G$ is a countable group and $(X,\pi)$ an isometric  Banach $G$-module. Then the following conditions are equivalent.
\begin{enumerate}
\item The space of $1$-coboundaries $B^1(G,X)$ is closed in $Z^1(G,X)$,
\item the isometric Banach $G$-module $(X/X^G,\ov \pi)$ does not have almost invariant unit vectors.
\end{enumerate}
\end{lemme}

\begin{proof}
(1)$\saa$(2): Observe first that $X\overset{{\partial^1}}\longrightarrow B^1(G,X)$ factors through a continuous linear bijection
$$
X/X^G\overset{\tilde{\partial^1}}\longrightarrow B^1(G,X).
$$
Suppose now that $B^1(G,X)$ is closed in $Z^1(G,X)$ and therefore is a Fr\'echet space. Then, by the open mapping theorem,  ${\tilde{\partial^1}}$ has a continuous inverse and hence there are a finite set $E\subseteq G$ and some $\eps>0$ so that $\norm{\ov x}_{X/X^G}\leqslant 1$ whenever $\max_{g\in E}\norm{x-\pi(g)x}\leqslant\eps$. By Lemma \ref{quotient displacement}, it follows that
$$
\norm{\ov x}_{X/X^G}
\leqslant 
\frac 1\eps\cdot\max_{g\in E}\norm{x-\pi(g)x}
\leqslant \frac 2\eps\cdot\max_{g\in E}\norm{\ov x-\ov\pi(g)\ov x}_{X/X^G}
$$
and thus $(X/X^G,\ov \pi)$ does not have almost invariant unit vectors.

(2)$\saa$(1):
Suppose conversely that $(X/X^G,\ov \pi)$ does not have almost invariant unit vectors. This means that there is some finite set $E\subseteq G$ and some $K$ so that 
$$
\norm{\ov x}_{X/X^G}\leqslant K\cdot\max_{g\in E}\norm{\ov x-\ov\pi(g)\ov x}_{X/X^G}\leqslant K\cdot\max_{g\in E}\norm{x-\pi(g)x}
$$
for all $x\in X$.
To see that $B^1(G,X)$ is closed in $Z^1(G,X)$, assume that $\partial^1x_n\Lim{n}\phi$ in $Z^1(G,X)$. By the above inequality, we see that $(\ov{x}_n)$ is Cauchy in $X/X_G$ and thus converges to some $\ov x\in X/X^G$, whereby $\phi=\partial^1x\in B^1(G,X)$.
\end{proof}

Recall that a group $G$ is said to have {\em property (T)} if every isometric Hilbert $G$-module with almost invariant unit vectors also has actual invariant unit vectors. There are various ways of generalising this to other classes of Banach spaces. We shall be following the formulation given in \cite{BFGM}, which however merits some clarification.

\begin{defi}
Let $\ku B$ be a class of Banach spaces. A group G is said to have {\em property ($\text{T}_{\ku B}$)} if, for every isometric Banach $G$-module $(X,\pi)$ with $X\in \ku B$, the isometric Banach $G$-module $(X/X^G,\ov\pi)$ does not have almost invariant unit vectors.
\end{defi}

The central observation regarding this definition is that, due to Lemma \ref{guichardet}, a countable group $G$ has property  ($\text{T}_{\ku B}$) if and only if 
$B^1(G,X)$ is closed in $Z^1(G,X)$ for all isometric Banach $G$-modules $(X,\pi)$ with $X\in \ku B$.

The next result is a partial generalisation to the setting of uniformly convex Banach spaces of an earlier result due to Y. Shalom \cite{shalom} (see also Theorem 3.2.1 \cite{bekka}).

\begin{prop}\label{shalom}
The following conditions are equivalent for a finitely generated group $G$ without $\Z$ as a quotient.
\begin{enumerate}
\item for every uniformly convex isometric Banach $G$-module $(X,\pi)$, we have 
$$
H^1(G,X)=\{0\},
$$
\item for every uniformly convex isometric Banach $G$-module $(X,\pi)$ so that $X^G=\{0\}$, we have 
$$
\ov H^1(G,X)=\{0\}.
$$ 
\end{enumerate}
Furthermore, these conditions imply that $G$ has property $(\text{T}_{\ku {UC}})$, where $\ku {UC}$ is the class of uniformly convex Banach spaces.
\end{prop}

\begin{proof}
That (1) implies (2) is obvious, so let us consider the converse. Assume that $(X,\pi)$ is a uniformly convex Banach $G$-module so that $H^1(G,X)\neq \{0\}$. Then we may pick some $\phi\in Z^1(G,X)$ so that the affine isometric action $G\overset \alpha\curvearrowright X$ given by $\alpha(g)x=\pi(g)x-\phi(g)$ has no fixed points on $X$. Fix also a finite generating set $\Sigma$ for $G$. Using a construction of M. Gromov and R. Schoen (see, for example, \cite{stalder} for details), we may produce another (necessarily affine) isometric action $G\overset \beta\curvearrowright Y$ on a uniformly convex Banach space so that
$$
\max_{g\in \Sigma}\norm{y-\beta(g)y}\geqslant 1
$$
for all $y\in Y$. Let $\sigma$ and $\psi$ denote respectively the linear and translation parts of $\beta$ and let $Y=Y^G\oplus Y_G$ be the associated Alaoglu--Birkhoff decomposition. Let also $Y\overset P\longrightarrow Y^G$ be the associated projection. Because $P$  is $G$-equivariant, we find that
$$
P\psi(gf)=P\big(\sigma(g)\psi(f)+\psi(g)\big)=\sigma(g)P\psi(f)+P\psi(g)=P\psi(f)+P\psi(g),
$$
that is, $P\psi$ is a homomorphism from $G$ into the additive group of the Banach space $Y^G$. However, as $G$ is finitely generated and does not have $\Z$ as a quotient, it follows that $P\psi$ is constantly $0$ and therefore that $\psi(g)\in Y_G$ for all $g\in G$. This implies that $Y_G$ is invariant under the action $G\overset \beta\curvearrowright Y$ and that $(Y_G,\sigma)$ is a uniformly convex isometric Banach $G$-module without invariant unit vectors for which 
$$
\ov H^1(G,Y_G)\neq \{0\}.
$$ 
So (2)$\saa$(1).

Finally, if, for every uniformly convex isometric Banach $G$-module $(X,\pi)$, we have $H^1(G,X)=\{0\}$, then $B^1(G,X)$ is closed in $Z^1(G,X)$, which, by Lemma \ref{guichardet}, implies that $G$ has property ($\text{T}_{\ku {UC}}$).
\end{proof}




\end{document}